\documentclass[12pt,reqno]{amsart}

\usepackage{amsmath,amsfonts,amsthm,graphicx,float,color}
\usepackage{cite}
\usepackage[hidelinks]{hyperref}
\usepackage{xcolor}
\hypersetup{
    colorlinks,
    linkcolor={blue!80!black},
    citecolor={blue!80!black},
    urlcolor={blue!80!black}
}
\usepackage{ulem}

\usepackage{cancel}
\usepackage[margin = 1in]{geometry}
\usepackage{mathrsfs}
\usepackage{mathabx}
\usepackage{dsfont}

\newcommand{\N}{\mathbb{N}}

\newcommand{\R}{\mathbb{R}}
\newcommand{\Z}{\mathbb{Z}}

\newcommand{\lp}{\left(}
\newcommand{\rp}{\right)}

\newcommand{\wt}{\widetilde}
\newcommand{\wh}{\widehat}

\newcommand{\sgn}{\textnormal{sgn}}

\newcommand{\supp}{\textnormal{supp\,}}

\newcommand{\Cone}[2]{\|#1\|_{C^1(#2)}}
\newcommand{\Conedot}[2]{\|#1\|_{\dot C^1(#2)}}

\DeclareMathOperator{\Spec}{Spec}

\DeclareMathOperator{\Cov}{Cov}
\DeclareMathOperator{\tr}{trace}
\DeclareMathOperator{\Hess}{Hess}
\DeclareMathOperator{\inj}{inj}
\renewcommand{\b}{\mathfrak{b}}
\newcommand{\Ce}{{C_{\bad}}}
\newcommand{\Ces}{{C_{\bad}^*}}

\makeatletter
\def\thmhead@plain#1#2#3{%
  \thmname{#1}\thmnumber{\@ifnotempty{#1}{ }\@upn{#2}}%
  \thmnote{ {\the\thm@notefont#3}}}
\let\thmhead\thmhead@plain
\makeatother

\newtheorem{thm}{Theorem}
\newtheorem{lem}{Lemma}
\newtheorem{prop}[lem]{Proposition}
\newtheorem{cor}[lem]{Corollary}
\newtheorem{rmk}[lem]{Remark}

\newtheorem{conjecture}[lem]{Conjecture}

\newcommand{\thmref}[1]{Theorem~\ref{#1}}
\newcommand{\secref}[1]{Section~\ref{#1}}
\newcommand{\lemref}[1]{Lemma~\ref{#1}}

\newcommand{\corref}[1]{Corollary~\ref{#1}}

\newcommand{\propref}[1]{Proposition~\ref{#1}}

\definecolor{bpurple}{RGB}{170,0,200}
\definecolor{jgreen}{RGB}{71,173,73}

\numberwithin{equation}{section}
\numberwithin{lem}{section}

\newcommand{\w}{{\rm{w}}}
\newcommand{\bad}{{\varepsilon}}

\renewcommand{\emph}[1]{{\it{#1}}}

\begin{document}

\title[Spectral function asymptotics on Zoll manifolds]{Asymptotics for the spectral function on Zoll manifolds}
\author[Y. Canzani]{Yaiza Canzani}
\email{\href{mailto:canzani@email.unc.edu}{canzani@email.unc.edu}}
\address{Department of Mathematics, UNC-Chapel Hill \\ CB\#3250
  Phillips Hall \\ Chapel Hill, NC 27599}

\author[J. Galkowski]{Jeffrey Galkowski}
\email{\href{mailto:j.galkowski@ucl.ac.uk}{j.galkowski@ucl.ac.uk}}
\address{Department of Mathematics, University College London\\ CB\#3250
  Union Building \\ London, United Kingdom WC1H 0AY}

\author[B. Keeler]{Blake Keeler}
\email{\href{mailto:bkeeler2015@gmail.com}{bkeeler2015@gmail.com}}
\address{Department of Mathematics and Statistics, McGill Univeristy \\ Montreal, Quebec Canada H3A 0B9}

\date{\today}

\begin{abstract}
Let $(M,g)$ be a Zoll manifold, i.e., a smooth, compact, Riemannian manifold without boundary all of whose geodesics are closed with a minimal common period $T$. The positive definite Laplace-Beltrami operator has eigenvalues $\{\lambda_j^2\}_j$ which cluster around $\nu^2_\ell$ for some sequence $\nu_\ell\to \infty$. This article is concerned with the number of $\lambda_j$ in a window of fixed size $\w$ around $\nu_\ell$, denoted by
$\mathcal{N}(\nu_\ell,\w):=\#\{j\,:\, \lambda_j\in[\nu_\ell-\w,\nu_\ell+\w]\}.$
When the set of trajectories with period smaller than $T$ has zero measure,  there is $c_{n}>0$, depending only on $n=\dim M$, such that 
$$
\mathcal{N}(\nu_\ell,\w) =c_n\textup{vol}_g(M)\nu_{\ell}^{n-1}+o(\nu_{\ell}^{n-1}),
$$ 
as $\ell \to \infty$.
However, for a general Zoll manifold this may not be the case. We show that, nevertheless, there is $N>0$, independent of $\ell$, such that 
$$
\sum_{j=0}^{N-1}\mathcal{N}(\nu_{\ell+j},\w)= c_nN\textup{vol}_g(M)\nu_{\ell}^{n-1}+o(\nu_{\ell}^{n-1}),
$$
as $\ell \to \infty$.
In addition to asymptotics for the counting function, we study the kernel of the spectral projector for the Laplacian, $\Pi_{\ell,\w}(x,y)$ onto the spectrum in ${\bigcup_{j=0}^{N-1}[\nu_{\ell+j}-\w,\nu_{\ell+j}+\w]}$. We show that for $x$ and $y$ in a shrinking neighborhood of a point with few loops of length smaller than $T$, $\Pi_{\ell,\w}(x,y)$ and its derivatives have the same asymptotics as those on the round sphere and flat torus.

\end{abstract}
\maketitle

\section{Introduction}

Let $(M,g)$ be a compact, Riemannian manifold without boundary and let $\Delta_g$ be the associated, negative definite,  Laplace-Beltrami operator. Denote the eigenvalues of $-\Delta_g$ by $0 = \lambda_0^2 < \lambda_1^2 \le \lambda_2^2\le\dotsm$  repeated according to multiplicity. 
For $I \subset \R$ let
$$\mathcal N(I):=\#\{j: \lambda_j\in I\}.$$
The Weyl Law states that 
\begin{equation}\label{e:Weyl}
\mathcal N([0, \lambda]) =(2\pi)^{-n}\text{vol}(\mathbb{B}_{n}) \text{vol}_g(M)\lambda^n + R(\lambda),
\end{equation}
where $R(\lambda)=\mathcal O(\lambda^{n-1})$ as $\lambda \to \infty$ \cite{Wey12, Lev53, Ava56, Hor68}.  This remainder term is sharp and is saturated, for example, on the round sphere, $\mathbb{S}^n$. However, when the set of closed geodesics has measure zero in $S^*M\!$, the remainder, $R(\lambda)$, can be improved to $o(\lambda^{n-1})$~\cite{DG75}. The improved remainder also allows for asymptotics on short windows: for $\w>0$,
\begin{equation}\label{e:Weyl2}
\mathcal{N}([\lambda-\w,\lambda+\w])=  2\w(2\pi )^{-n} \textnormal{vol}(\mathbb{S}^{n-1})\text{vol}_g(M)\lambda^{n-1} + o_{\w}(\lambda^{n-1}).
\end{equation}

A {\it{Zoll manifold $(M,g)$}} is a smooth, compact, Riemannian manifold without boundary such that all of its geodesics are periodic with a common period. This is a rich class of manifolds that includes compact rank one symmetric spaces. Indeed, while the most well known example of a Zoll manifold is the round sphere, $\mathbb{S}^2$, the moduli space of Zoll metrics on $\mathbb S^2$ is infinite dimensional~\cite{Gu:76}. 

It is well known that, like on the sphere of radius $\frac{T}{2\pi}$, the eigenvalues of $-\Delta_g$ on a Zoll manifold with minimal common period $T$ are strongly clustered near the sequence
\begin{equation}\label{e:Zoll_eigs}
\nu_\ell := \frac{2\pi}{T}\Big(\ell + \frac{\mathfrak{a}}{4}  \Big), \quad \ell = 0,1,2,\dotsc,
\end{equation}
where $\mathfrak{a}$ is the common Maslov index of the closed  geodesics \cite{Wei75, Wei77, DG75, CV79, Cha80}. The remainder estimate $R(\lambda)=O(\lambda^{n-1})$ is saturated on any Zoll manifold. As proved in \cite{DG75}, if the set of periodic trajectories with period $<T$ has zero measure, then a modified version of~\eqref{e:Weyl2} which takes into account the clustering holds: for all $0<\w<\frac{2\pi}{T}$
\begin{equation}\label{e:remainderWeylZoll1}
\mathcal{N}([\nu_{\ell}-\w,\nu_{\ell}+\w])=\frac{2\pi}{T}(2\pi )^{-n}\textnormal{vol}(\mathbb{S}^{n-1})\textnormal{vol}_g(M)\nu_\ell^{n-1}+o_{\w}(\nu_\ell^{n-1}).
\end{equation}

A Zoll manifold all of whose geodesics do not self-intersect before time $T$ is called simply closed, or $SC_T$, in the language of~\cite{Zel97}. In particular,~\eqref{e:remainderWeylZoll1} (and much stronger estimates) hold for an $SC_T$ manifold. Many Zoll manifolds are $SC_T$ manifolds. Indeed, all known smooth Zoll metrics on simply connected manifolds yield $SC_T$ manifolds.  However, as far as the authors are aware, the only topological manifold on which all Zoll metrics are known to be $SC_T$ metrics is $\mathbb{S}^2$ ~\cite{GrGr:81}. 

The example of the disjoint union of two simply connected Zoll manifolds with different, rationally related minimal common periods shows that~\eqref{e:remainderWeylZoll1} cannot hold without additional assumptions. One would like to know whether all connected, smooth, compact, Zoll manifolds satisfy~\eqref{e:remainderWeylZoll1}. However, we do not know whether such metrics must have a zero measure set of periodic geodesics of period $<T$. Our first theorem shows that, nevertheless, an analog of~\eqref{e:remainderWeylZoll1} holds for \emph{every} Zoll manifold  provided that one is willing to sum over a finite number of small windows. In what follows, $\inj(M)$ denotes the injectivity radius of $M$.
\begin{thm}\label{t:main}
Let $(M,g)$ be a smooth, compact, Zoll manifold of dimension $n \ge 2$ with minimal common period $T > 0$. Then, there is an integer $0<N_0<\frac{T}{\inj (M)}$ such that for all $N\geq N_0$ and $0<\w<\frac{2\pi}{T}$, 
\begin{equation}\label{e:remaindeWeyl}
\begin{aligned}\sum_{j=0}^{N-1}\mathcal{N}([\nu_{\ell+j}-\w,\nu_{\ell+j}+\w])&=\frac{2\pi N}{T}(2\pi )^{-n}\textnormal{vol}(\mathbb{S}^{n-1})\textnormal{vol}_g(M)\nu_\ell^{n-1}+o_{\w}(\nu_\ell^{n-1}),
\end{aligned}
\end{equation}
as $\ell \to \infty$.
\end{thm}
We note that $N_0$ can be taken to be the smallest integer such that the set of trajectories with period smaller than $T/N_0$ has zero Liouville measure on $S^*\!M$. Indeed, for any $SC_T$ manifold, $N_0=1$, and this recovers ~\eqref{e:remainderWeylZoll1}.

We also note that the estimate~\eqref{e:remaindeWeyl} captures the majority of the eigenvalues in the window $[\nu_{\ell}-\w,\nu_{\ell+N-1}+\w]$, since
\begin{equation}
\label{e:comparison}
\sum_{j=0}^{N-1}\mathcal{N}([\nu_{\ell+j}-\w,\nu_{\ell+j}+\w])=\mathcal{N}([\nu_{\ell}-\w,\nu_{\ell+N-1}+\w])+o(\nu_{\ell}^{n-1}).
\end{equation}
In fact, substantially stronger estimates than~\eqref{e:comparison} hold (see e.g.~\cite{DG75,CV79}).

Next, we describe a refinement of Theorem~\ref{t:main} with applications to the theory of random waves. For this, we let $\{\varphi_j\}_{j=0}^\infty$ be an orthonormal basis of $L^2(M)$ such that 
\begin{equation}\label{e:efxs}
-\Delta_g \varphi_j = \lambda_j^2\varphi_j, \quad j = 0,1,2,\dotsc,
\end{equation}
and for $I\subset \mathbb{R}$ consider the orthogonal projection operator
\[
\Pi_I : L^2(M)\to \bigoplus\limits_{\lambda_j\in I}\ker(\Delta_g + \lambda_j^2).
\]
 The Schwartz kernel of $\Pi_I$ takes the form 
\begin{equation}\label{e:pi}
\Pi_I(x,y) = \sum\limits_{\lambda_j\in I} \varphi_j(x)\overline{\varphi_j(y)}, \qquad x,y \in M.
\end{equation}
The operator $\Pi_I$ plays a crucial role in studying both Weyl laws, since $\tr\Pi_I=\mathcal{N}(I)$, and monochromatic random waves.

We study the asymptotics as $\lambda\to \infty$ of spectral projectors of the form $\Pi_{I_\lambda}(x,y)$, where $I_\lambda$ is an interval, centered at $\lambda$, with length uniformly bounded from above and below. These spectral projectors appear as the covariance kernels of monochromatic random waves (see \eqref{e:cov}). The asymptotics of $\Pi_{I_\lambda}(x,y)$ are intimately connected to the dynamics of the geodesic flow on $(M,g)$.

The most classical random wave studies occur on the round sphere, $\mathbb{S}^{n}$, and flat torus, $\mathbb{T}^n$. In the case of the sphere, 
$$
\lambda_\ell^2=\ell(\ell+n-1)\,\quad\, \ell=0,1,\dots,
$$
and it is known that, with $\nu_\ell:=\ell+\frac{n-1}{2}$ and $0<\w<1$ for $x,y\in \mathbb{S}^{n}$,  with $d_g(x,y)\leq r_\ell$ and $\lim_{\ell\to \infty}r_\ell=0$, 
\begin{equation}
\label{e:sphere}
\Pi_{\{\lambda_\ell\}}(x,y)=\Pi_{[\nu_\ell-\w,\nu_\ell+\w]}(x,y)= \frac{\nu_\ell^{n-1}}{(2\pi)^{n/2}}\frac{J_{\frac{n-2}{2}}(|\nu_\ell d_g(x,y)|)}{(\nu_\ell d_g(x,y))^{\frac{n-2}{2}}} {+o(\nu_{\ell}^{n-1})}, \quad \ell\to \infty.
\end{equation}
Here, we write $d_g(x,y)$ for the Riemannian distance between $x$ and $y$ and $J_\alpha$ for the Bessel function of the first kind with index $\alpha$. 

Despite the fact that the dynamics of the geodesic flow on the $n$-dimensional flat torus are dramatically different than those on the sphere, we also have for $\w>0$, $x,y\in \mathbb{T}^{n}$ with $d_g(x,y)\leq r_\nu$, and $\lim_{\nu \to \infty}r_\nu =0$, 
\begin{equation}
\label{e:torus}
\Pi_{[\nu-\w,\nu+\w]}(x,y)= \frac{2\w\nu^{n-1}}{(2\pi)^{n/2}}\frac{J_{\frac{n-2}{2}}(|\nu d_g(x,y)|)}{(\nu d_g(x,y))^{\frac{n-2}{2}}}+{o(\nu^{n-1})}, \qquad \nu\to \infty.
\end{equation}
Indeed, one expects that the local behavior of $\Pi_{I_\lambda}$ is, in some sense, universal.

\begin{conjecture}
\label{conj1}
Let $(M,g)$ be a smooth, compact, Riemannian manifold of dimension $n$  without boundary and $x_0\in M$. Then, there exist $C>0$, a sequence $\nu_\ell\to \infty$, and a sequence $0<\w_\ell<C $ such that for any positive sequence $r_\ell \to 0$
\begin{equation}
\label{e:conjAsymp}
\sup_{x,\,y\in  B(x_0,r_\ell)}\Bigg|\frac{\Pi_{[\nu_\ell-\w_\ell,\nu_\ell+\w_\ell]}(x,y)}{\mathcal N([\nu_\ell-\w_\ell,\nu_\ell+\w_\ell])}- \frac{(2\pi)^{n/2}}{\textnormal{vol}(\mathbb{S}^{n-1})}\frac{J_{\frac{n-2}{2}}(|\nu_\ell d_g(x,y)|)}{(\nu_\ell d_g(x,y))^{\frac{n-2}{2}}}\Bigg|=o(1), \qquad \ell\to \infty,
\end{equation}
with analogous bounds for $C^k$ norms of the spectral projector. 
\end{conjecture}

Observe that for any $0<\w<1$ on the round sphere, we have
$$
\mathcal N([\nu_\ell-\w,\nu_\ell+\w])=\frac{\textnormal{vol}(\mathbb{S}^{n-1})}{(2\pi)^n}\nu_\ell^{n-1}+o(\nu_\ell^{n-1}),
$$
and on the torus we have
$$
\mathcal N([\nu_\ell-\w,\nu_\ell+\w])=2\w\frac{\textnormal{vol}(\mathbb{S}^{n-1})\nu_\ell^{n-1}}{(2\pi)^n}+o(\nu_\ell^{n-1}).
$$
Hence, in both cases, \eqref{e:sphere} and \eqref{e:torus} yield
$$\frac{\Pi_{[\nu_\ell-\w,\nu_\ell+\w]}(x,y)}{\mathcal N([\nu_\ell-\w,\nu_\ell+\w])}=\frac{(2\pi)^{\frac{n}{2}}}{\textnormal{vol}(\mathbb{S}^{n-1})}\frac{J_{\frac{n-2}{2}}(|\nu_\ell d_g(x,y)|)}{(\nu_\ell d_g(x,y))^{\frac{n-2}{2}}}+o(1),$$
and the conjecture holds in these examples.

In~\cite{CH15,CH18}, Canzani--Hanin showed that the asymptotics~\eqref{e:conjAsymp} hold whenever $x$ is a non-self focal point. That is, the set of directions $\xi \in S_{x}^*M$ that generate a geodesic loop that returns to $x$ has Liouville measure zero. As for the flat torus, in the case of non-self focal points, one can take any sequence $\nu_\ell\to \infty$ and $\w_\ell=1$.

On a Zoll manifold every point is, in some sense, the opposite of non-self focal.
Because of the sphere-like clustering of the spectrum, it is too much to hope that~\eqref{e:conjAsymp} holds for any choice of $\nu_\ell\to \infty$  and, as in the case of the Weyl law, we should instead work with spectral projectors for a well chosen sequence $\nu_\ell$. In particular, we take $\nu_\ell$ as in~\eqref{e:Zoll_eigs}. 

Our goal is to show that Conjecture~\ref{conj1} holds at certain points on Zoll manifolds. As discussed above, many Zoll manifolds of period $T$ are $SC_T$ manifolds. However, it is possible that some may have geodesics with shortest periods $T/N$ for some $N>1$ or closed geodesics of length $T$ that are not simple (i.e. that pass over the same base point more than once). Indeed, the (albeit trivial) example of the disjoint union of two Zoll manifolds with rationally related periods shows that, at least in principle, there may be a large set of closed geodesics of period smaller than $T$. However, these must have period $T/N$ for some fixed $N$.

In order to handle this type of situation, we formulate our next theorem in a way that allows for large sets of loops at times rationally related to $T$, as well as a zero-volume set of loops with non-rationally related looping time. 
For $N \in \mathbb N$, $T>0$, $\bad>0$, define the sets 
$$
K^{T}_N := \Big\{\frac{p}{q}T:\,p,q\in\N,\,1\le p < q\le N\Big\},\qquad K^{T}_{N,\bad}:=((-\bad,\bad)+K^T_N)\cup (-\infty,0)\cup (T,\infty).$$ 
Let $(M,g)$ be a smooth Zoll manifold of dimension $n \ge 2$ with minimal common period $T > 0$ and $\varphi_t:S^*M \to S^*M$ denote the geodesic flow for time $t$. Fix a metric on $T^*M$, and define
$$
\mathscr{L}_{N,\bad}(x_0):=\Big\{ \rho\in S^*_{x_0}M\,:\, \bigcup_{t\in (K^T_{N,\bad})^c} \varphi_t (B_{S^*M}(\rho,\bad))\cap S^*_{B(x_0,\bad)}M \neq \emptyset\Big\}.
$$
Note that a direction $\rho \in S^*_{x_0}M$ is in $\mathscr{L}_{N,\bad}(x_0)$ if there is a time $t \in [0,T]$ that is at least $\bad$-far from every element in $K_N^T$ such that $\varphi_{t}(B_{S^*M}(\rho,\bad))\cap T^*_{B(x_0,\bad)}M\neq \emptyset$.

Our next result gives pointwise estimates on the spectral projector near any $x_0\in M$ such that $\mu_{S^*M}\left(\mathscr{L}_{N,\bad}(x_0)\right)\overset{\bad\to 0}{\longrightarrow}0$.

\begin{thm}\label{t:main_thm_2}
Let $(M,g)$ be a smooth, compact, Zoll manifold of dimension $n \ge 2$ with minimal common period $T > 0$ and $\nu_\ell$ defined in~\eqref{e:Zoll_eigs}. For $x,y\in M$ and $\w>0$, define 
\begin{equation}\label{e:remainder}
R_{N,\w}(\ell;x,y):=\frac{1}{N}\sum\limits_{j=0}^{N-1}\Pi_{[\nu_{\ell+j}-\w,\nu_{\ell+j}+\w]}(x,y)-\frac{2\pi }{T} \frac{\nu_\ell^{n-1}}{(2\pi)^{n/2}}\frac{J_{\frac{n-2}{2}}(\nu_\ell d_g(x,y))}{(\nu_\ell d_g(x,y))^{\frac{n-2}{2}}}
\end{equation}
Let $N > 0$ and $x_0\in M$ such that
\begin{equation}\label{e:loopset_assumption}
\lim\limits_{\bad\to 0}\mu_{S^*M}\left(\mathscr{L}_{N,\bad}(x_0)\right) = 0.
\end{equation}
Then, for any $0<\w<\frac{2\pi}{T}$ and $\alpha,\beta \in \mathbb N^n$,
\begin{equation}\label{e:scaling_limit}
\lim\limits_{\delta\to 0^+}\limsup\limits_{\ell\to\infty}\sup\limits_{x,y\in B(x_0,\delta)}\left|\nu_\ell^{1-n-|\alpha|-|\beta|}\partial_x^\alpha\partial_y^\beta R_{N,\w}(\ell;x,y)\right| = 0.
\end{equation}
\end{thm}
For $SC_T$ manifolds, there are no sub-periodic loops and hence~\eqref{e:loopset_assumption} holds with $N=1$. We also note that, for $SC_T$ manifolds, the on-diagonal version of this result, without derivatives, was also proved in \cite[Theorem 2]{Zel97}.


As a corollary of Theorems~\ref{t:main} and~\ref{t:main_thm_2} we obtain that Conjecture~\ref{conj1} holds whenever~\eqref{e:loopset_assumption} holds on a Zoll manifold. 
\begin{cor}
\label{c:conjTrue}
Let $(M,g)$ be a smooth, connected, Zoll manifold of dimension $n \ge 2$ with minimal common period $T > 0$. Fix any $N > 0$ and $x_0\in M$ such that~\eqref{e:loopset_assumption} holds.
Then Conjecture~\ref{conj1} holds at $x_0$. 
\end{cor}

As discussed briefly before, a motivation for proving Theorem \ref{t:main_thm_2} is its application to the theory of random waves on manifolds.  A \emph{monochromatic random wave} on $(M,g)$ is a Gaussian random field of the form
\begin{equation}\label{e:psi}
\psi_{\lambda,\w}(x) 
:= \big(\mathcal{N}([\lambda-\w,\lambda+\w])\big)^{-\frac{1}{2}}\sum\limits_{\lambda_j\in[\lambda-\w,\lambda+\w]}a_j \varphi_j(x),
\end{equation}
where the $a_j$ are i.i.d. standard Gaussian random variables and the $\varphi_j$ are the eigenfunctions in \eqref{e:efxs}.

Monochromatic random waves were created to model eigenfunction behavior. Although $\psi_{\lambda,\w}$ is not an actual eigenfunction, it is expected to behave like one. (For a careful account of the history, see \cite{Can20, Wig22} and references there.) In particular, much research has been dedicated to understanding the behavior of the zero sets and critical points of random waves.  The corresponding features of deterministic eigenfunctions are very difficult to study, and their analysis  becomes much more tractable for the monochromatic random counterparts. 

The statistics of $\psi_{\lambda,\w}$ are completely determined by the associated two-point correlation function 
\begin{equation}\label{e:cov}
K_{\lambda,\w}(x,y) := \Cov\lp\psi_{\lambda,\w}(x),\psi_{\lambda,\w}(y)\rp = \frac{\Pi_{[\lambda-\w,\lambda+\w]}(x,y)}{\mathcal{N}([\lambda-\w,\lambda+\w])}, \qquad x,y\in M.
\end{equation}
Most research is typically done on the round sphere or the flat torus since $K_{\lambda,\w}$ is well understood for these spaces \cite{BMW20, CW17, KKW13, BCW19, Cam19, CMW16, NS09, RW08}. Studying features like the zero sets and critical points of $\psi_{\lambda,\w}$ relies on having asymptotics for $K_{\lambda,\w}(x,y)$ when $x,y \in B(x_0, \frac{1}{\lambda})$ with $x_0$ fixed. Although treating $K_{\lambda,\w}$ on general manifolds is quite challenging, Conjecture~\ref{conj1} would imply that, when the eigenvalue intervals defining the sum in \eqref{e:psi} are appropriately chosen,
 \begin{equation}\label{e:kernelConv}
 \lim\limits_{\ell\to\infty}\sup\limits_{|u|,|v|\le r_\ell}\left|\partial_u^\alpha\partial_v^\beta \left(K_{\nu_\ell,\w}\Big(\exp_{x_0}\big(\tfrac{u}{\nu_\ell}\big),\exp_{x_0}\big(\tfrac{v}{\nu_\ell}\big)\Big) - \frac{(2\pi)^{n/2}}{\textnormal{vol}(S^*\!M)}\frac{J_{\frac{n-2}{2}}(|u-v|)}{(|u-v|)^{\frac{n-2}{2}}} \right)\right|=0.
 \end{equation}
    Here, $\exp_{x_0}:T_{x_0}^*M \to M$ denotes the exponential map with footpoint at $x_0$.  Corollary~\ref{c:conjTrue} shows that for appropriately chosen intervals these asymptotics do, in fact, hold at points on a Zoll manifold where~\eqref{e:loopset_assumption} is satisfied.

%

Results about Conjecture~\ref{conj1} yield corresponding asymptotics for the covariance function of monochromatic random waves. Indeed, for a general manifold $(M,g)$, when the interval in \eqref{e:psi} is $[\lambda-\frac{1}{2}, \lambda+\frac{1}{2}]$, the asymptotics from~\cite{CH15, CH18} show that~\eqref{e:kernelConv} holds when the point $x_0$ is non self-focal. In the case where $(M,g)$ has no conjugate points~\cite{Keeler2020} (or more generally there are `very' few loops, see~\cite{CanzaniGalkowski2020}), the asymptotics in \eqref{e:kernelConv} hold at every point with a logarithmic improvement on the rate of decay to $0$.

In the language of Nazarov-Sodin \cite{NS16}, if the asymptotics in \eqref{e:kernelConv} hold at every $x_0\in M$,  then the random waves $\psi_{\lambda, \w}$ have translation invariant local
limits. For ensembles with such translation invariant local limits, Zelditch \cite{Zel09}, Nazarov-Sodin \cite{NS16},
Sarnak-Wigman \cite{SW19}, Gayet-Welschinger \cite{GW16}, Canzani-Sarnak \cite{CS19}, Canzani-Hanin \cite{CH20} as well
as others, prove detailed results on non-integral statistics of the nodal sets of
random waves. Such nodal set statistics include the number of connected components, Betti numbers, and topological types.

\subsection{Comments on the proof}
Since our long term goal is to approach Conjecture~\ref{conj1}, we aim to implement a method that uses only the dynamical information obtained from the fact that $(M,g)$ is Zoll. In particular, to prove Theorem \ref{t:main_thm_2}, we avoid using the fact that one can find $Q$, a pseudodifferential operator of order $-1$, such that $\sqrt{-\Delta_g+Q}$ has spectrum contained in $\cup_{\ell}\{\nu_\ell\}$~\cite{CV79}. This extra structure was used in~\cite{Zel97} to obtain a full asymptotic expansion of $\Pi_{[\nu_{\ell}-\w,\nu_\ell+w]}(x,x)$ in the $SC_T$ case.

Using standard Tauberian arguments, the analysis reduces to understanding the singularities of $e^{it\sqrt{-\Delta_g}}(x,y)$ for $t\in[-\sigma^{-1},\sigma^{-1}]$ with $\sigma\to 0$ very slowly as $\lambda\to \infty$. These singularities are located at times $t$ when there is a geodesic loop from $x$ to $y$. We analyze these singularities in three pieces:  1) We use the periodicity of the flow to study the singularities near $kT$, $k\in \mathbb{Z}$. 2) We show that zero measure sets of loops do not to contribute to the main asymptotics using methods similar to those in~\cite{SoZe:02,SoToZe:11,CH18,Wy:19,WyYaZe:22}. 3) By summing of the windows $[\nu_{\ell+j}-\w,\nu_{\ell+j}+\w]$, $j=0,\dots, N-1$, we are able to incorporate the function $\sin(\pi Nt/T)/\sin(\pi t/T)$ into the amplitude multiplying $e^{-it\sqrt{-\Delta_g}}$ (see \eqref{e:lambda*}). Using this extra structure, we then show that even positive measure sets of loops at times $\frac{jT}{N}$ do not contribute to the leading term of the asymptotics. Note that, when considering asymptotics for the counting function Theorem \ref{t:main}, only periodic trajectories need to be analyzed. Thus, since periodic trajectories must have minimal period $T/N$ for some $N$, we need no extra assumption to obtain asymptotics for the counting function.

\subsection{Organization of the paper} 
We begin in Section \ref{s:loops} by analyzing the implications of the assumption \eqref{e:loopset_assumption}, namely that it allows us to construct a pair of microlocal cutoffs which localize near, and respectively away from, the measure zero set of geodesics which have have looping times outside of $K_N$. Section \ref{s:smooth} proceeds with an analysis of the asymptotic contributions of the smooth spectral projector microlocalized away from all subperiodic loops. This is complemented by Section \ref{s:subperiodic} which studies the contributions near both types of subperiodic looping times: those which lie near $K_N$ (which may have positive measure), and those which do not (and must therefore have zero measure by assumption). We take a brief detour in Section \ref{s:onDiag} to prove some estimates on the spectral projector restricted to the diagonal, which are necessary for estimating the difference between the smooth and rough projectors. These on-diagonal estimates do not depend on the subperiodic loops assumption and slightly generalize similar results in \cite{DG75}. In Section \ref{s:proofs}, we assemble the pieces produced in the previous sections to complete the proof of \thmref{t:main_thm_2}. Theorem \ref{t:main} is proved in Section \ref{s:tmain}.

\bigskip
\noindent\textsc{Acknowledgements.}  Y.C. was supported by NSF CAREER Grant DMS-2045494 and NSF Grant DMS-1900519.  J.G. is grateful to the EPSRC for partial funding under Early Career Fellowship EP/V001760/1 and Standard Grant EP/V051636/1. B.K. was supported by postdoctoral fellowships through CRM-ISM and AARMS.

\section{Microlocalization Near Subperiodic Loops}\label{s:loops}

In this section, we discuss a technical construction that is essential for the asymptotic analysis in the subsequent parts of the proof of \thmref{t:main_thm_2}. Our assumption on  $\mathscr L_{N,\bad}(x_0)$ allows for the existence of subperiodic loops with looping times which do not lie near $K_N$, as long as the set of such loops is sufficiently small in measure. It is therefore crucial to construct a pair of pseudodifferential cutoffs which localize near and away from these loops. This idea is analogous to the constructions done in \cite{SoggeZelditch2002} and \cite{CH15}, although our procedure is somewhat different because we cannot rely solely on the upper semicontinuity of the reciprocal of the return-time function.  

To begin the construction, we first prove the following lemma, which shows that that the assumption \eqref{e:loopset_assumption} implies a more concrete fact in local coordinates. 
Below, $m$ denotes the Lebesgue measure on $S^{n-1}$.

\begin{lem}\label{lem:measure}
Fix $T>0$, let $K\Subset \mathbb{R}$ be a closed set, and define $K_\bad:=((-\bad,\bad)+K)\cup (-\infty,0)\cup (T,\infty)$. Let 
$$
A_{\bad}(x_0):=\Big\{ \rho\in S^*_{x_0}M\,:\, \bigcup_{t\in K_{\bad}^c} \varphi_t (B_{S^*M}(\rho,\bad))\cap S^*_{B(x_0,\bad)}M \neq \emptyset\Big\}.
$$
Fix $x_0\in M$, let $\gamma$ be a diffeomorphism from a neighborhood of $x_0$ into $\R^n$, and set
\begin{equation}\label{e:L_def}
L_{\bad}(x_0):=\{ \xi\in S^{n-1}\,:\, \exists x \in B(x_0, \bad), \; t\in K_\bad, \,  \Pi_M(\varphi_t(\gamma^{-1}(x),(\partial\gamma(x))^t\xi))\in B(x_0, \bad)\}.
\end{equation}
Then,
\begin{equation}\label{e:MeasureControl}
\lim_{\bad \to 0^+} \mu_{S^*_{x_0}M}(A_\bad(x_0))=0 \quad \Rightarrow \quad \lim_{\bad \to 0^+} m(L_\bad(x_0))=0.
\end{equation}
\end{lem}

\begin{proof}
Let $\alpha>0$ to be chosen small and consider $\{\xi_j\}_{j=1}^{N_\bad}$ an $\alpha\bad$-maximal separated set in $S^{n-1}$. Then there is $\mathfrak{D}>0$, depending only on $n$, and $\{\mathcal{J}_\ell\}_{\ell=1}^{\mathfrak{D}}$ such that
\begin{equation}\label{e:covering}
\begin{gathered}
S^{n-1}\subset \bigcup_{j=1}^{N_\bad}B(\xi_j, \alpha \bad),
\qquad   \{1,\dots, N_\bad\}=\bigcup_{\ell=1}^{\mathfrak{D}}\mathcal{J}_\ell,\\
 B(\xi_j,10\alpha\bad)\cap B(\xi_k ,10\alpha\bad)=\emptyset,\, i\neq k,\,i,k\in \mathcal{J}_\ell.
\end{gathered}
\end{equation}

First, we claim there exists $\alpha_0>0$ such that if $\alpha<\alpha_0$, then
\begin{equation}\label{e:canFindEp}
B(\xi_j,\alpha\bad)\cap L_{\alpha\bad}(x_0)\neq \emptyset \quad  \Rightarrow \quad  B(\iota(\xi_j),\tfrac{1}{2}\bad)\subset A_\bad(x_0),
\end{equation}
where $\iota:S^{n-1}\to S^*_{x_0}M$ is the map $\iota(\xi):=(x_0,(\partial\gamma(0))^t\xi)$.

Indeed, let $\eta\in L_{\alpha\bad}(x_0)$ with $|\eta-\xi_j|<\alpha\bad$. Then, there are $|x|<\alpha \bad$ and $t\in K_\bad^c$ such that 
$$
\varphi_t(\gamma(y),(\partial \gamma(y))^t\eta)\in B(x_0,\alpha\bad). 
$$
Now, $d(\gamma(y),x_0)\leq C\alpha \bad$, and $d(\iota(\eta),(\partial\gamma(y))^t\eta)<C\alpha \bad$.
Similarly, $d(\iota(\xi_j),\iota(\eta))<C\alpha\bad$, so that 
$$
(\gamma(y),(\partial\gamma(y))^t\eta)\in B(\iota(\xi_j),C\alpha \bad). 
$$
Choosing $\alpha_0<\frac{1}{2C}$, then implies that for any $\rho\in B(\iota(\xi_j),\frac{1}{2}\bad)\cap S^*_{x_0}M$, 
$$
(\gamma(y),(\partial\gamma(y))^t\eta)\in B(\rho, \bad)
$$
which, in turn, implies that 
$
B(\iota(\xi_j),\tfrac{1}{2}\bad)\subset A_\bad(x_0).
$
This proves the claim in \eqref{e:canFindEp}.

Notice that there is a $C>0$ such that for all $\ell$ and any $\rho\in S^*_{x_0}M$
\begin{equation}
\label{e:notTooMuchOverlap}
|\{j\in \mathcal{J}_\ell\,:\, B(\iota(\xi_j),\alpha\bad)\cap B(\rho, \bad)\}|\leq C \alpha^{1-n}. 
\end{equation}
Now, define
$$
\begin{gathered}
\mathcal{I}:=\{ j\in\{1,\dots,N_\bad\}\,:\, B(\xi_j,\alpha\bad)\cap L_{\alpha\bad}(x_0)\neq \emptyset\}, \\
\mathcal{I}_\ell:=\{ j\in\mathcal{J}_\ell\,:\, B(\xi_j,\alpha\bad)\cap L_{\alpha\bad}(x_0)\neq \emptyset\}.
\end{gathered}
$$
Then, by~\eqref{e:notTooMuchOverlap},
\begin{align*}
\mu_{S^*_{x_0}M}(A_\bad(x_0))
&\geq \max_{\ell}\Big|\bigcup_{j\in \mathcal{I}_\ell}B(\rho_j,\tfrac{1}{2}\bad)\Big|
\geq c\alpha^{n-1}\max_{\ell}|\mathcal{I}_\ell|\bad^{n-1}\\
&\geq c\alpha^{n-1}\bad^{n-1}|\mathcal{I}|/\mathfrak{D}
\geq c m(L_{\alpha\bad}(x_0))/\mathfrak{D}.
\end{align*}
\end{proof}

With \lemref{lem:measure} in hand, we construct the desired pseudodifferential cutoffs. Fix a point $x_0\in M$ which satisfies \eqref{e:loopset_assumption} and choose a diffeomorphism $\gamma$ from a neighborhood $U$ of $x_0$ into $\R^n.$ Then, by \lemref{lem:measure}, the measure of the sets $L_{\bad}(x_0)$ tends to 0 as ${\bad}\to 0^+.$  Thus, for any ${\bad} > 0$, there exists an open set $\mathcal O_{\bad}\subset S^{n-1}$ such that $L_{\bad}(x_0)\subseteq \mathcal O_{\bad}$ and $m(\mathcal O_{\bad}) < {\bad}/2.$ Hence, we can find some $\wt b_{\bad}\in C^\infty(S^{n-1})$ that is identically 1 on $\mathcal O_{\bad}$ and zero outside of a slightly larger open set $V_{\bad}$ with $m(V_{\bad}) < {\bad}.$ Now, let $\chi\in C^\infty(M)$ be supported in the coordinate neighborhood $U$ and equal to 1 on a slightly smaller neighborhood, and choose some $\beta\in C^\infty(\R)$ which vanishes on a neighborhood of 0 and is equal to 1 outside $[-1/2,1/2]$. Then, setting
\[b_{\bad}(x,\xi) = \chi(x)\beta(|\xi|)\wt b_{\bad}(\xi/|\xi|), \qquad c_{\bad}(x,\xi):= 1 - b_{\bad}(x,\xi)\]
we define the pseudodifferential operators $B_{\bad}$, $C_{\bad}$ 
\[B_{\bad} f(x) = \frac{1}{2\pi}\int\limits e^{i\langle\gamma(x)-\gamma(y),\xi\rangle}b_{\bad}(x,\xi)f(y)\,dy\,d\xi,\]
\[C_{\bad} f(x) = \frac{1}{2\pi}\int\limits e^{i\langle\gamma(x)-\gamma(y),\xi\rangle}c_{\bad}(x,\xi)f(y)\,dy\,d\xi.\]
Note that 
\begin{equation}\label{e:adjoints}
B_{\bad}+ C_{\bad}=I, \qquad B_{\bad}^*+C_{\bad}^*=I.
\end{equation}

Observe that
\begin{gather}
\label{e:suppc}
\supp c_{\bad}\cap \overline{L_{\bad}(x_0)},\\ 
\sup_{x\in M, r\geq \frac{1}{2}}\|1-c_{\bad}(x,\cdot)\|_{L^1(\mathbb{S}^{n-1})}\leq \bad.\label{e:l1C}
\end{gather}

\section{Analysis of the Smoothed Projector away from Subperiodic Loops}\label{s:smooth}
By the construction in the preceding section, for any fixed $\bad > 0,$ we have a microlocal partition of unity near $x_0$ in the form of $B_{\bad}$ and $\Ce$. By \eqref{e:adjoints},
\begin{equation}\label{e:localized}
\frac{1}{N}\sum\limits_{j=0}^{N-1}\Pi_{[\nu_{\ell+j}-\w,\nu_{\ell+j}\lambda+\w]}(x,y) = \frac{1}{N}\sum\limits_{j=0}^{N-1}\Pi_{[\nu_{\ell+j}-\w,\nu_{\ell+j}\lambda+\w]}(B_{\bad}^* + \Ces)(x,y)
\end{equation}
 Since $B_{\bad}^*$ has small microsupport, we expect the contribution from this term to be negligible from the perspective of the asymptotics. We prove this rigorously in \secref{s:subperiodic}. The bulk of our analysis is dedicated to studying the $\Ces$ term. In fact, we will study a smoothed version of this object, which involves a convolution with a suitably chosen Schwartz-class function.

We introduce $\rho\in\mathscr S(\R)$ with the property that $\wh\rho$ is supported in $[-2,2]$ and equal to one on $[-1,1]$. Then, for any $\sigma>0$, let $\rho_{\sigma}(\mu) = \frac{1}{\sigma}\rho(\mu/\sigma)$, so that
\begin{equation}\label{e:rho_ep}
\wh\rho_{\sigma}(t) = \wh\rho(\sigma t)
\end{equation}
is supported in $[-2/\sigma,2/\sigma]$ and equal to one on $[-1/\sigma,1/\sigma].$ 
The goal of this section is to study the asymptotic behavior of 
\[\frac{1}{N}\sum\limits_{j=0}^{N-1}\rho_{\sigma}\ast \Pi_{[\nu_{\ell+j}-\w,\nu_{\ell+j}+\w]}\Ces.\] 
This is done in  \propref{p:smooth_prop} below. In preparation for this result, in Section \ref{s:halfwave} we first rewrite $\rho_{\sigma}\ast \Pi_{[\lambda-\w,\lambda+\w]}$ in terms of the kernel of the half wave operator and its singularities. Later, in Section \ref{s:Qprop}, we find the asymptotic behavior of the kernel when localized to each singularity. We finally state and prove  \propref{p:smooth_prop} which combines these estimates to obtain asymptotics for the full projector.

\subsection{Singularities of the half-wave operator}\label{s:halfwave}
To study the smoothed projector, for any $\w,\sigma > 0$ we define
$
\psi_{\sigma}(\mu):=\rho_{\sigma}\ast \mathds 1_{[-\w,\w]}(\mu) ,
$
which is Schwartz-class and has Fourier transform
\begin{equation}\label{e:psi_defn}
\wh\psi_{\sigma}(t) = \wh\rho_{\sigma}(t)\frac{2\sin(t\w)}{t}.
\end{equation}
Then, if $U_t(x,y)$ denotes the kernel of the half-wave operator $U_t = e^{-it\sqrt{-\Delta_g}}$, we have
\begin{equation}\label{e:smooth_projector}
\rho_{\sigma}\ast \Pi_{[\lambda-\w,\lambda+\w]}\Ces(x,y) = \frac{1}{2\pi}\int\limits_{-\infty}^\infty e^{it\lambda}\wh\psi_{\sigma}(t)U_t\Ces(x,y)\,dt
\end{equation}
for all $\bad>0$, by Fourier inversion. Note that on the left-hand side of \eqref{e:smooth_projector}, the convolution is taken with respect to the $\lambda$ variable. From \cite{DG75}, we have that $U_t$ is a Fourier integral operator of class $I^{-\frac{1}{4}}(\R\times M,M;\mathcal C)$, where the canonical relation $\mathcal C$ is given by
\begin{align}\label{e:canonical_relation}
\begin{split}
\mathcal C &= \left\{\big((t,\tau),(x,\xi),(y,\eta)\big):\, (t,\tau)\in T^*\R\setminus \{0\},\right.\\
& \hspace{0.5in} \left.(x,\xi),(y,\eta)\in T^*M\setminus \{0\},\, \tau + |\xi_g| = 0,\,(x,\xi) = \Phi^t(y,\eta)\right\},
\end{split}
\end{align}
where $\Phi^t:T^*M\to T^*M$ denotes the geodesic flow.
For any $\lambda > 0$, 
\begin{align}
\label{e:lambda*}
\sum\limits_{j=0}^{N-1}\rho_\sigma\ast\Pi_{[\lambda+\frac{2\pi j}{T}-\w,\lambda+\frac{2\pi j}{T}+\w]}\Ces(x,y) & = \frac{1}{2\pi}\int\limits_{-\infty}^\infty \sum\limits_{j=0}^{N-1}e^{it(\lambda+\frac{2\pi j}{T})}\wh\psi_{\sigma}(t)U_t\Ces(x,y)\,dt\notag\\
& = \frac{1}{2\pi}\int\limits_{-\infty}^\infty e^{it(\lambda+\frac{(N-1)\pi}{T})}\frac{\sin\lp\frac{\pi N t}{T}\rp}{\sin\lp\frac{\pi t}{T}\rp}\wh\psi_{\sigma}(t)U_t\Ces(x,y)\,dt,
\end{align}
where the final equality follows from the Dirichlet kernel identity 
$\sum\limits_{j=0}^{N-1}e^{ijx} = e^{i(N-1)x/2}\frac{\sin(Nx/2)}{\sin(x/2)}.$
Later, we will set $\lambda = \nu_\ell$, for $\nu_\ell$ defined as in \eqref{e:Zoll_eigs}. We have
\[
\frac{1}{2\pi}\int\limits_{-\infty}^\infty e^{it(\lambda+\frac{(N-1)\pi}{T})}\frac{\sin\lp\frac{\pi N t}{T}\rp}{\sin\lp\frac{\pi t}{T}\rp}\wh\psi_{\sigma}(t)U_t\Ces(x,y)\,dt= \mathscr A_{\bad}(\lambda,\sigma;x,y) + \mathscr B_{\bad}(\lambda,\sigma;x,y)
\]
for
\begin{align}
\label{e:A_defn}\mathscr A_{\bad}(\lambda,\sigma;x,y) &:= \frac{1}{2\pi}\int\limits_{-\infty}^\infty e^{it(\lambda+\frac{(N-1)\pi}{T})}\frac{\sin\lp\frac{\pi N t}{T}\rp}{\sin\lp\frac{\pi t}{T}\rp}\wh\psi_{\sigma}(t)U_t\Ces(x,y)\sum\limits_{k\in\Z}\wh\rho(t-kT)\,dt,\\
\label{e:B_defn}\mathscr B_{\bad}(\lambda,\sigma;x,y)&:=\frac{1}{2\pi}\int\limits_{-\infty}^\infty e^{it(\lambda+\frac{(N-1)\pi}{T})}\frac{\sin\lp\frac{\pi N t}{T}\rp}{\sin\lp\frac{\pi t}{T}\rp}\wh\psi_{\sigma}(t)U_t\Ces(x,y)\lp 1 - \sum\limits_{k\in\Z}\wh\rho(t-kT)\rp\,dt.
\end{align}
We can think of $\mathscr A$ and $\mathscr B$ as being localized near to and away from times which are integer multiples of $T$, respectively. 
We first consider $\mathscr A_{\bad}(\lambda,\sigma;x,y)$. Changing variables, $t \mapsto t+kT$, 
\begin{equation}\label{e:kT_shift}
\begin{aligned}
&\mathscr A_{\bad}(\lambda,\sigma;x,y)=\\
&=\sum\limits_{k\in\Z}\frac{e^{ikT(\lambda+\frac{(N-1)\pi}{T})}}{2\pi}\int\limits_{-\infty}^\infty e^{it(\lambda+\frac{(N-1)\pi}{T})}{{(-1)^{(N-1)k}}}\frac{\sin\lp\frac{\pi N t}{T}\rp}{\sin\lp\frac{\pi t}{T}\rp}\wh\psi_{\sigma}(t+kT)\wh\rho(t)U_{t+kT}\Ces(x,y)\,dt \\
&= \sum\limits_{k\in\Z}\frac{e^{ikT\lambda }}{2\pi}\mathcal F^{-1}_{t\mapsto \lambda}\lp \wh f_k(t) U_{t+kT}\Ces(x,y)\rp,
\end{aligned}
\end{equation}
where we define 
\begin{equation}\label{e:f_k}
\wh f_k(t) = {e^{i t\frac{(N-1)\pi}{T}}}\frac{\sin\lp\frac{\pi N t}{T}\rp}{\sin\lp\frac{\pi t}{T}\rp}\wh\psi_{\sigma}(t+kT)\wh\rho(t),
 \end{equation}
 and $\mathcal F^{-1}_{t\mapsto\lambda}$ is the inverse Fourier transform mapping $t$ to $\lambda$. Then, we can use that ${U_{s}\varphi_j = e^{-is\lambda_j}\varphi_j}$ to obtain
\begin{align*}
\mathcal F^{-1}_{t\mapsto\lambda}\lp\wh f_k(t)U_{t+kT}\Ces(x,y)\rp &=\mathcal F^{-1}_{t\mapsto\lambda}\lp \wh f_k(t)\sum\limits_{j=0}^\infty e^{-i\lambda_j(t+kT)}\varphi_j(x)\overline{\Ce\varphi_j(y)}\rp\\
& = f_k\ast\lp\sum\limits_{j=0}^\infty\delta(\lambda-\lambda_j)e^{-ikT\lambda_j}\varphi_j(x)\overline{\Ce\varphi_j(y)}\rp\\
& = f_k\ast \partial_\lambda\lp\sum\limits_{\lambda_j\le\lambda}\varphi_j(x)\overline{\Ce U_{-kT}\varphi_j(y)}\rp\\
& = \partial_\lambda\lp f_k\ast\Pi_{[0,\lambda]} U_{kT}\Ces(x,y)\rp. 
\end{align*}
Therefore, if $d(x,y)\leq \delta$, \eqref{e:kT_shift} yields
\begin{equation}
\mathscr A_{\bad}(\lambda,\sigma;x,y)  = \sum\limits_{k\in\Z}e^{ikT\lambda}\partial_\lambda(f_k\ast\Pi_{[0,\lambda]} U_{kT}\Ces(x,y)).
\end{equation}
By \cite[page 53]{DG75},
with $\mathfrak{a}$ as in \eqref{e:Zoll_eigs} and
\begin{equation}\label{e:b}
\b := \frac{\pi\mathfrak{a}}{2T},
\end{equation}
we have that $U_t - e^{i\b T}U_{t+T}$ is a Fourier integral operator of one order lower than $U_t$, namely $-\frac{1}{4} - 1$. In particular, we have that $U_0 - e^{i\b T}U_T$ is a pseudodifferential operator of order $-1$, and 
\[U_{0} - e^{ik\b T}U_{kT}\in \Psi^{-1}(M),\]
for any $k\in \Z$. Since $U_0$ is the identity map, we can write
\begin{equation}\label{Q_k_definition}
U_{kT} = {e^{-ik\b T}(I + Q_k)}
\end{equation}
for $Q_k\in\Psi^{-1}(M)$ with polyhomogeneous symbol. Thus, we obtain
\begin{equation}\label{e:smooth_sum}
\mathscr A_{\bad}(\lambda,\sigma;x,y) = \sum\limits_{k\in\Z}e^{ikT(\lambda-\mathfrak b)}\partial_\lambda(f_k\ast\Pi_{[0,\lambda]}(I+Q_k)\Ces)(x,y).
\end{equation}
Therefore, we must determine the asymptotic behavior of  $\partial_\lambda(f_k\ast\Pi_{[0,\lambda]}(I+Q_k)\Ces)$.

\begin{rmk}\label{r:finite}\textnormal{
Note that for each fixed $\sigma,\delta > 0$, the $\wh f_k$ are identically 0 for sufficiently large $k$. Therefore, the sum in \eqref{e:smooth_sum} is finite for each $\sigma,\delta> 0$.
}
\end{rmk}

\subsection{Pseudodifferential perturbations of the Spectral Projector}\label{s:Qprop}

The goal of this section is to find the asymptotic behavior of
$$
\partial_\lambda(f_k\ast \Pi_{[0,\lambda]}(I+Q_k)\Ces)(x,y)
$$
for each $k$. We are interested in working with points $x,y\in M$ for which  $d_g(x,y)$ is small. Therefore, we will assume that we work with coordinates $y=(y_1, \dots, y_n)$ on $M$ and dual coordinates $(\xi_1, \dots, \xi_n)$ on $T_y^*M$. The Riemannian volume form in these coordinates takes the form $\sqrt{|g_y|}dy$, where $|g_y|$ denotes the determinant of the matrix representation of $g(y)$. We also define the function
\[
\Theta(x,y):=|\text{det}_g D_{\exp_x^{-1}(y)}\exp_x|,
\]
where the subscript $g$ means that we use the metric to choose an orthonormal basis on $T_{\exp_x^{-1}(y)}(T_xM)$ and $T_y^*M$ (c.f. \cite[Chapter 2, Proposition C.III.2]{BGM}). The determinant is then independent of the choice of such a basis. We note that $\Theta(x,y)=\sqrt{|g_x|}$ in normal coordinates centered at $y$.

If $\xi \in T_y^*M$ is represented as $\xi=r\omega$ with $(r, \omega)\in (0,+\infty)\times S_y^*M$, then we endow $S_y^*M$ with the measure $d\omega$ such that $d\xi=r^{n-1}  d\omega dr$. 

\begin{rmk}\label{r:Bessel}\textnormal{
We note that $d\omega$ is not a coordinate invariant measure, but it behaves like a density in $y$ under changes of coordinates. Thus, $d\omega$ should be regarded as a measure taking values in the space of densities on $M$. Despite this, we note that for $v\in \R^n$
\begin{equation}\label{e:bessel}
\frac{1}{(2\pi)^{n/2}} \frac{J_{\frac{n-2}{2}}(|v|)}{|v|^{\frac{n-2}{2}}}=\frac{1}{(2\pi)^n}\int_{S^{n-1}}e^{i \langle v, \omega \rangle} \, d\sigma_{_{\!\mathbb S^{n-1}}}\!(\omega).
\end{equation} 
Hence
\[
\frac{1}{(2\pi)^n}\int\limits_{S_y^*M}e^{i\lambda\langle\exp_y^{-1}(x),\omega\rangle_{g}}\frac{d\omega}{\sqrt{|g_y|}} =\frac{1}{(2\pi)^{n/2}}\frac{J_{\frac{n-2}{2}}(|\lambda d_g(x,y)|)}{(\lambda d_g(x,y))^{\frac{n-2}{2}}},
\]
and the right hand side is clearly coordinate invariant. Here, we used that $d\omega=|g_y|^{1/2}d\sigma_{_{\!S^{n-1}}}$  and that in local coordinates $\langle\exp_y^{-1}(x),\omega\rangle_{g}=\langle g_y^{-1/2}\exp_y^{-1}(x),g_y^{-1/2}\omega\rangle_{_{\R^n}}$
with $g_y^{-1/2}\omega\in \mathbb S^{n-1}$ and $|g_y^{-1/2}\exp_y^{-1}(x)|_{_{\R^n}}=d_g(x,y)$.
}
\end{rmk}

\begin{prop}\label{Q_prop}
Let $(M,g)$ be a compact, smooth Riemannian manifold of dimension ${n\ge 2}$ without boundary.  Let $C$ and $Q$ be pseudodifferential operators with polyhomogeneous symbols $c$ and $q$
of orders 0 and $-1$, respectively.  Fix $\delta \leq \frac{1}{2}\textnormal{inj}(M,g)$. Then, for each pair of multi-indices $\alpha,\beta \in \mathbb N^n$, there exist constants $C_1,C_2,\mu_0>0$, such that for any function $f\in C^\infty(\R)$ with $\wh f$ smooth and compactly supported, and any $x,y\in M$ with $d_g(x,y)\le \delta$ we have
\begin{align*}
&\Theta^{\frac{1}{2}}(x,y)\partial_\mu\lp  f\ast\Pi_{[0,\mu]}(I+ Q)C\rp(x,y) \\
&\hskip 1.5in = \frac{\mu^{n-1}\wh f(0)}{(2\pi)^{n}}\int\limits_{S_y^*M}e^{i\mu\langle\exp_y^{-1}(x),\omega\rangle_{g_y}}c(y,\omega) \frac{d\omega}{\sqrt{|g_y|}} + R(\mu,x,y),
\end{align*}
with
\begin{align}\label{Q_prop_R_est}
\begin{split}
  \sup\limits_{d_g(x,y)\le \delta}| \partial_x^\alpha\partial_y^\beta R(\mu,x,y)| &\le C_1\delta \|\partial_t {\wh f}\|_{L^\infty([-\delta,\delta])}\mu^{n-1+|\alpha|+|\beta|} + C_2\mu^{n-2+|\alpha|+|\beta|}
\end{split}
\end{align}
for all $\mu \ge \mu_0$. 
Here, $C_1$ is independent of $\delta$, $Q$ and $f$.
\end{prop}

\begin{proof}
We prove the statement first in the case where $\alpha = \beta = 0$. Observe that %
\begin{equation}\label{e:smooth_projector_Q}
\partial_\mu\lp f\ast \Pi_{[0,\mu]} (I+Q)C\rp(x,y) = \frac{1}{2\pi}\int\limits_{-\infty}^\infty e^{it\mu} \wh f(t)U_t (I+Q)C(x,y)\,dt.
\end{equation}
Using the parametrix for $U_t$ constructed in \cite[Proposition 8]{CH15}, we have that if $d_g(x,y)\le \frac{1}{2}\textnormal{inj}(M,g)$, then
\begin{equation}\label{e:parametrix}
U_t(x,y) = \frac{\Theta^{-\frac{1}{2}}(x,y)}{(2\pi)^n}\int\limits_{T_y^*M}e^{i\langle\exp_y^{-1}(x),\xi\rangle_{g_y} - it|\xi|_{g_y}} A(t,y,\xi)\frac{d\xi}{\sqrt{|g_y|}}
\end{equation}
modulo smoothing kernels, for some symbol $A\in S^0$ with a polyhomogeneous expansion $A\sim\sum\limits_{j=0}^\infty A_{-j}$. In particular, $A_0(t,y,\xi) \equiv 1$ for all $t$, and when $t = 0$, $A_{-j}(0,y,\xi) = 0$ for all $j\ge 1$. Since $C$ and $Q$ are pseudodifferential, we can use the same parametrix construction to write
\begin{equation}\label{e:UQ}
U_t(I+Q)C(x,y) = \frac{\Theta^{-\frac{1}{2}}(x,y)}{(2\pi)^n}\int\limits_{T_y^*M}e^{i\langle\exp_y^{-1}(x),\xi\rangle_{g_y} - it|\xi|_{g_y}} D(t,y,\xi)\frac{d\xi}{\sqrt{|g_y|}}
\end{equation}
for some $D\in S^{0}$. Note that since the principal symbol of $U_t$ is identically 1 and $C,Q$ are pseudodifferential, the principal symbols of $U_tC$ and $U_tQC$ are each independent of $t$. At $t = 0,$ we have $U_0C = C$ and $U_0QC = QC$, and hence the principal symbol of $U_tC$ is $c_0(y,\xi)$ for all $t$. Furthermore, since the subprincipal symbol of $C$ is identically zero and all lower order terms of $A$ vanish at $t = 0$, we have that the symbol of $U_t(I+Q)C$ satisfies
\begin{equation}\label{e:D_symbol}
D(t,y,\xi) - c_0(y,\xi) -D_{-1}(t,y,\xi)\in S^{-2}.
\end{equation}
 where  $D_{-1} \in S^{-1}$ is homogeneous degree $-1$.
From \eqref{e:smooth_projector_Q} and \eqref{e:UQ} we obtain
\begin{align}
\begin{split}\label{e:smooth_projector_parametrix}
& \partial_\mu(f\ast \Pi_{\mu}(I+ Q)C)(x,y)\\
& \hskip 1in= \frac{\Theta^{-\frac{1}{2}}(x,y)}{(2\pi)^{n+1}}\int\limits_{-\infty}^\infty\int\limits_{T_y^*M} e^{it\mu}e^{i\langle\exp_y^{-1}(x),\xi\rangle_{g_y} - it|\xi|_{g_y}}\wh f(t)D(t,y,\xi)\frac{d\xi dt}{\sqrt{|g_y|}} + \mathcal O(\mu^{-\infty}).
\end{split}
\end{align}
To control the integral on the right-hand side above, we change variables via $\xi\mapsto \mu r\omega$ for $(r,\omega)\in \R^+\times S_y^*M$, which yields that the LHS of \eqref{e:smooth_projector_parametrix} is
\begin{equation}\label{e:before_stat}
\frac{\mu^{n}}{(2\pi)^{n+1}}\int\limits_{-\infty}^\infty\int\limits_0^\infty \wh f(t)e^{i\mu t(1-r)} r^{n-1}\lp\int\limits_{S_y^*M}e^{i\mu r\langle\exp_y^{-1}(x),\omega\rangle_{g_y}}D(t,y,\mu r\omega)\frac{d\omega}{\sqrt{|g_y|}}\rp\,dr\,dt.
\end{equation}
Noting that since the phase is nonstationary for $r\ne 1$ we may introduce a cutoff function $\zeta\in C_c^\infty(\R)$ which is equal to one on a neighborhood of $r = 1$, and supported in $[\frac{1}{2},\frac{3}{2}].$ This results in an error which is $\mathcal O(\mu^{-\infty})$ as $\mu\to\infty.$ 

Let $S(t,y,\xi) = c_0(y,\xi) + D_{-1}(t,y,\xi)$ be the first two terms in the polyhomogeneous expansion of $D$. Since $D - S$ is a symbol of order $-2$, we have $|D(t,y,\mu r\omega) - S(t,y,\mu r\omega)| \le C\mu^{-2}$ uniformly for all $t,y$. Combining this fact with an application of stationary phase in $(t,r)$, we see that the LHS of \eqref{e:smooth_projector_parametrix} is equal to
\begin{equation}\label{e:n-3_error}
\frac{\mu^n}{(2\pi)^{n+1}}\int\limits_{-\infty}^\infty\int\limits_{-\infty}^\infty \wh f(t)e^{i\mu t(1-r)} r^{n-1}\zeta(r)\lp\int\limits_{S_y^*M}e^{i\mu r\langle\exp_y^{-1}(x),\omega\rangle_{g_y}}S(t,y,\mu r\omega)\,\frac{d\omega}{\sqrt{|g_y|}}\rp\,dr\,dt + \mathcal O(\mu^{n-3}),
\end{equation}
where $\zeta\in C_c^\infty(\mathbb R)$ is a cut-off function that is equal to $1$  near $r=1$ and vanishes for $r\notin [\tfrac{1}{2}, \tfrac{3}{2}]$. 
Notice that by homogeneity in the fiber variable, we have that for any $(y,\eta)\in T^*M$,
\begin{equation}\label{e:spherical_decomp}
\int\limits_{S_y^*M}e^{i\langle\eta,\omega\rangle_{g_y}}S(t,y,\mu r\omega)\frac{d\omega}{\sqrt{|g_y|}} = \int\limits_{S_y^*M}e^{i\langle\eta,\omega\rangle_{g_y}}\lp c_0(y,\omega)+ \frac{1}{\mu r}D_{-1}(t,y,\omega) \rp\,\frac{d\omega}{\sqrt{|g_y|}}.
\end{equation}
Then, following the proof of \cite[Theorem 1.2.1]{SoggeBook2014}, there exist smooth functions $a_{\pm}\in C^\infty(T^*M)$ and $b_\pm\in C^\infty(\R\times T^*M)$ such that 
\begin{equation}\label{e:sogge_1}
\int\limits_{S_y^*M}e^{i\langle\eta,\omega\rangle_{g_y}}c_0(y,\omega)\,\frac{d\omega}{\sqrt{|g_y|}} = \sum\limits_{\pm}e^{\pm i|\eta|_{g_y}}a_\pm(y,\eta),
\end{equation}
and 
\begin{equation}\label{e:sogge_2}
\int\limits_{S_y^*M}e^{i\langle\eta,\omega\rangle_{g_y}}D_{-1}(t,y,\omega) \frac{d\omega}{\sqrt{|g_y|}}= \sum\limits_{\pm}e^{\pm i|\eta|_{g_y}}b_{\pm}(t,y,\omega),
\end{equation}
satisfying the estimates
\begin{equation}\label{e:ab_estimates}
|\partial_\eta^\gamma a_\pm(y,\eta)|\le C_\gamma (1+|\eta|_{g_y})^{-\frac{n-1}{2}-|\gamma|},\qquad |\partial_t^k\partial_\eta^\gamma b_\pm(t,y,\eta)|\le C_{\gamma,k}(1 + |\eta|_{g_y})^{-\frac{n-1}{2} -|\gamma|}
\end{equation}
for any multi-index $\gamma$, any integer $k\ge 0$, and some constants $C_\gamma,\, C_{\gamma,k}$ which are independent of $t,\,y,$ and $\eta.$ Therefore, by \eqref{e:smooth_projector_Q}, \eqref{e:parametrix},  \eqref{e:UQ}, \eqref{e:sogge_1} and \eqref{e:sogge_2}
\begin{equation}\label{e:stat_phase_form}
\partial_\mu(f\ast \Pi_{[0,\lambda]}(I+ Q)C)(x,y)
=\frac{\mu^{n}}{(2\pi)^{n+1}}\sum\limits_{\pm}\int\limits_{\R}\int\limits_0^\infty e^{i\mu\psi_\pm(t,r,x,y)}g_\pm(t,r,x,y,\mu)\,dr\,dt,
\end{equation}
where $\psi_\pm(t,r,x,y)= t(1-r)\pm rd_g(x,y)$ and 
\begin{equation}\label{e:g}
g_\pm(t,r,x,y,\mu) = r^{n-1}\zeta(r)\wh f(t)\lp a_\pm(y,\mu r\exp_y^{-1}(x)) + \frac{1}{\mu r}b_\pm(t,y,\mu r \exp_y^{-1}(x))\rp.
\end{equation}
Observe that for any fixed $x,y\in M$, the critical points of $\psi_\pm$ occur at $(t_c^\pm,r_c^\pm) = (\pm d_g(x,y),1)$, and that 
\[
\det\lp\textnormal{Hess}\,\psi_\pm(t_c^\pm,r_c^\pm,x,y)\rp = 1.
\]
Therefore, by the method of stationary phase, we see that 
\begin{equation*}
\label{F_bound}
\begin{aligned}
&\partial_\mu(f\ast \Pi_{[0,\lambda]}(I+ Q)C)(x,y)\\
&\qquad=\frac{\mu^{n-1}}{(2\pi)^n}\sum\limits_{\pm}e^{\pm i\mu d_g(x,y)}\lp g_\pm(t_c^\pm,r_c^\pm,x,y,\mu) -\frac{i}{\mu}\partial_r\partial_t g_\pm(t_c^\pm,r_c^\pm,x,y,\mu)\rp  + \mathcal{O}(\mu^{n-3}).
\end{aligned}
\end{equation*}
From \eqref{e:g} and \eqref{e:ab_estimates} we have that 
\begin{align*}
\left|\partial_r\partial_t g_\pm(t_c^\pm,r_c^\pm,x,y,\mu)\right|&\le C_1|\partial_t\wh f(\pm d_g(x,y))| + \frac{C_2}{\mu}\lp|\wh f(\pm d_g(x,y))| + |\partial_t\wh f(\pm d_g(x,y))|\rp\\
& \le C_1\|\partial_t\wh f\|_{L^\infty([-\delta,\delta])} + \frac{C_2}{\mu}\|\wh f\|_{C^1([-\delta,\delta])},
\end{align*}
and we remark that $C_1$ is independent of $Q$ due to the definition of $a_\pm$.
Therefore,
\begin{align}
&\Theta^{\frac{1}{2}}(x,y)\partial_\mu\lp f\ast \Pi_{[0,\mu]} (I+Q)C\rp(x,y) \notag\\
&\qquad= \frac{\mu^{n-1}}{(2\pi)^n}\sum\limits_{\pm}e^{\pm i\mu d_g(x,y)}\wh f(\pm d_g(x,y))\lp a_\pm(y,\mu \exp_y^{-1}(x)) + \frac{1}{\mu} b_\pm(t_c^\pm,y,\mu\exp_y^{-1}(x))\rp \notag\\
&\qquad \;\; + R_1(\mu,x,y), 
\end{align} 
where 
\[
\sup\limits_{d_g(x,y)\le \delta}|R_1(\mu,x,y)| \le C_1\Conedot{\wh f}{[-\delta,\delta]}\mu^{n-2} + C_2\Cone{\wh f}{[-\delta,\delta]}\mu^{n-3} + \mathcal O(\mu^{n-3}),
\]
with $C_1$ independent of $Q$.
Next, let us Taylor expand $\wh f$ near 0, which yields
\[
\wh f(\pm d_g(x,y)) = \wh f(0) \pm d_g(x,y)\partial_t\wh f(s_\pm) 
\]
for some $s_\pm$ between $0$ and $\pm d_g(x,y).$ Combining this with the fact that
\begin{equation}\label{e:sogge_int}
\sum\limits_{\pm}e^{\pm i\mu d_g(x,y)} a_\pm(y,\mu\exp_y^{-1}(x)) =  \int\limits_{S_y^*M}e^{i\mu\langle\exp_y^{-1}(x),\omega\rangle}c_0(y,\omega)\,\frac{d\omega}{\sqrt{|g_y|}},
\end{equation}
we obtain
\begin{align}\label{e:R1_R2}
&\Theta^{\frac{1}{2}}(x,y)\partial_\mu\lp \wh f\ast\Pi_{[0,\mu]}(I+Q)C\rp(x,y) \notag\\
&\qquad  = \frac{\mu^{n-1}\wh f(0)}{(2\pi)^{n}}\lp \int\limits_{S_y^*M}e^{i\mu\langle\exp_y^{-1}(x),\omega\rangle_{g_y}}c_0(y,\omega)\,\frac{d\omega}{\sqrt{|g_y|}}+ \sum\limits_{\pm}e^{\pm i\mu d_g(x,y)}b_\pm(t_c^\pm,y,\mu\exp_y^{-1}(x))\rp \notag\\
&\qquad \;\;  + R_1(\mu,x,y) + R_2(\mu,x,y),
\end{align}
where $R_1$ is as above, and $R_2$ satisfies
\[\sup\limits_{d_g(x,y)\le \delta}|R_2(\mu,x,y)| \le  \delta\|{\partial_t \wh f}\|_{L^\infty([-\delta,\delta])}(C_0\mu^{n-1}+C_1\mu^{n-2})\]
for some $C_0>0$ which is independent of $Q$ and $C_1>0$. Next, we Taylor expand 
\[b_\pm(t_c^\pm,y,\mu\exp_y^{-1}(x)) = b_\pm(0,y,\mu\exp_y^{-1}(x)) \pm d_g(x,y)\partial_tb_\pm(s_\pm',y,\mu\exp_y^{-1}(x))\]
for some $s_\pm'$ between 0 and $t_c^\pm = \pm d_g(x,y)$. Recalling \eqref{e:ab_estimates}, we have that
\[|\partial_t b_\pm(s_\pm,y,\mu\exp_y^{-1}(x))| \le C_2(1 + \mu d_g(x,y))^{-\frac{n-1}{2}},\]
since $|s_\pm| \le d_g(x,y)$. Therefore, we obtain
\begin{align}\label{e:R3}
&\frac{\mu^{n-2}\wh f(0)}{(2\pi)^m}\sum\limits_{\pm}e^{\pm i\mu d_g(x,y)}b_\pm(t_c^\pm,y,\mu\exp_y^{-1}(x)) \notag\\
& \hskip 0.8in= \frac{\mu^{n-2}\wh f(0)}{(2\pi)^n}\int\limits_{S_y^*M}e^{i\mu\langle\exp_y^{-1}(x),\omega\rangle} D_{-1}(0,y,\omega)\,\frac{d\omega}{\sqrt{|g_y|}} + R_3(\mu,x,y), 
\end{align}
where
\[
\sup\limits_{d_g(x,y)\le \delta}|R_3(\mu,x,y)| \le C_2\delta\wh f(0)\mu^{n-2},
\]
after potentially increasing $C_2.$ Therefore, we have that \eqref{e:R1_R2} and \eqref{e:R3} yield
\begin{align*}
&\Theta^{\frac{1}{2}}(x,y)\partial_\mu\lp \wh f\ast\Pi_{[0,\mu]} (I+Q)C\rp(x,y)\\
 &\hskip 1.5in = \frac{\mu^{n-1}\wh f(0)}{(2\pi)^{n}}\int\limits_{S_y^*M}e^{i\mu\langle\exp_y^{-1}(x),\omega\rangle_{g_y}}c_0(y,\omega)\,\frac{d\omega}{\sqrt{|g_y|}}+ \wt R(\mu,x,y),
\end{align*}
where $\wt R$ satisfies
\begin{align*}
 \sup\limits_{d_g(x,y)\le \delta}|\wt R(\mu,x,y)| &\le C_1\delta\Conedot{\wh f}{[-\delta,\delta]}\mu^{n-1} +C_2\Conedot{\wh f}{[-\delta,\delta]}\mu^{n-2}\\
 & \hskip 0.5in + C_3 \delta\wh f(0)\mu^{n-2} + C_4 \Cone{\wh f}{[-\delta,\delta]}\mu^{n-3} +\mathcal O(\mu^{n-3}),
\end{align*}
for some $C_1,C_2,C_3,C_4>0$, with $C_1$ independent of $\delta,f,$ and $Q$. This completes the proof in the case where $\alpha = \beta = 0$.

To include derivatives in $x,y$, we observe that
\[\partial_x^\alpha\partial_y^\beta e^{i\langle\exp_y^{-1}(x),\xi\rangle} = \mathcal O(|\xi|^{|\alpha| + |\beta|})\]
as $|\xi|\to \infty$. Therefore, we can repeat the preceding argument where the orders of the symbols involved are increased by at most $|\alpha| + |\beta|$ to obtain the desired result.

\end{proof}

\subsection{Asymptotics for $\mathscr A_{\bad}(\lambda,\sigma;x,y)$}\label{s:rhoPi}
With \propref{Q_prop} in hand, we are equipped to prove the main result of this section, namely the asymptotic behavior of $\mathscr A_{\bad}(\lambda,\sigma;x,y)$, which accounts for the contributions near each multiple of the period $T$. In particular, we set $\lambda = \nu_\ell $ for $\ell = 0,1,2,\dotsc.$ Then, we can define
\begin{equation}\label{e:R_ep}
R_{\bad}(\ell,\sigma; x,y):= \mathscr A_{\bad}(\nu_\ell,\sigma;x,y) - \frac{2\pi N}{T}\cdot\frac{\nu_\ell^{n-1}}{(2\pi)^{n}}\int\limits_{S_y^*M}e^{i\nu_\ell\langle\exp_y^{-1}(x),\omega\rangle_g}\frac{d\omega}{\sqrt{|g_y|}}.
\end{equation}
\begin{prop}\label{p:smooth_prop}
Let $(M,g)$ be a smooth Zoll manifold with minimal common period $T>0.$ Fix $0 < \w < \frac{2\pi}{T}$. Let $\mathscr A_{\w,\bad}$ as in~\eqref{e:A_defn} with $C_{\bad}$ satisfying~\eqref{e:l1C}.  Then, for any multi-indices $\alpha,\beta \in \mathbb N^n$, 
\[
\lim\limits_{\sigma\to 0^+}\lim_{\bad\to 0^+}\lim\limits_{\delta\to 0^+}\limsup\limits_{\ell\to \infty}\sup\limits_{d_g(x,y)\le \delta}\left|\frac{1}{\nu_\ell^{n-1+|\alpha|+|\beta|}}\partial_x^\alpha\partial_y^\beta R_{\bad}(\ell,\sigma; x,y)\right| = 0.
\]
\end{prop}

\begin{proof}

Fix two multi-indices $\alpha,\beta \in \mathbb N^n$. First, note that for $\b$ as in \eqref{e:b} we have that for all $k\in\Z$
\[
e^{ikT(\b - \nu_\ell)} = e^{ikT(-2\pi\ell/T)} = e^{-2\pi ik\ell} = 1.
\]

Combine \eqref{e:smooth_sum} with \propref{Q_prop} to obtain
\begin{align}\label{e:triple_sum} 
\mathscr A_{\bad}(\lambda,\sigma;x,y)
&=\nu_\ell^{n-1}A_{\bad}(\ell,x,y)\sum\limits_{k\in\Z}\wh f_k(0) + \nu_\ell^{n-2}\sum\limits_{k\in\Z}\wh f_k(0)W_{k,\bad}(\ell,x,y)+\sum\limits_{k\in\Z} R_k(\ell,x,y),  
\end{align}
where
\begin{align}
A_{\bad}(\ell,x,y) 
&= \frac{1}{(2\pi)^n\Theta^{\frac{1}{2}}(x,y)}\int\limits_{S_y^*M}e^{i\nu_\ell\langle\exp_y^{-1}(x),\omega\rangle_g}c_{\bad}^0(y,\omega)\frac{d\omega}{\sqrt{|g_y|}}, \label{e:L_definition}\\
W_{k,\bad}(\ell,x,y) 
&= \frac{1}{(2\pi)^n\Theta^{\frac{1}{2}}(x,y)}\int\limits_{S_y^*M}e^{i\nu_\ell\langle\exp_y^{-1}(x),\omega\rangle_g}c_{\bad}^0(y,\omega)\sigma(Q_k)(y,\omega)\frac{d\omega}{\sqrt{|g_y|}}, \label{e:W_k_definition}
\end{align}
and $R_k$ satisfies
\begin{align*}
\sup\limits_{d_g(x,y)\le \delta}|\partial_x^\alpha\partial_y^\beta R_k(\ell,x,y)| & \le  C_1\delta {\|\partial_t\wh f_k\|}_{L^\infty([-\delta,\delta])}\nu_\ell^{n-1+|\alpha|+|\beta|} +C_2\nu_\ell^{n-2+|\alpha|+|\beta|}
\end{align*}
with $C_1$ independent of $\delta$ and $k$.  Recalling that the summation in $k$ is actually finite and that $\sup_{\{\sigma > 0,\delta< 1, k\in\Z\}}\|\partial_t\wh f_k\|_{L^\infty([-\delta,\delta])}<\infty$ (see Remark \ref{r:finite}) we have that if we define 
\[F_{\delta,\sigma}(x,y,\ell):=\frac{1}{\nu_\ell^{n-1+|\alpha|+|\beta|}}
\left(\nu_\ell^{n-2}\sum\limits_{k\in\Z}\wh f_k(0)\partial_x^\alpha\partial_y^\beta W_k(\nu_\ell,x,y) + \sum\limits_{k\in\Z} \partial_x^\alpha\partial_y^\beta R_k(\nu_\ell,x,y)\right),\]
then we have for each fixed $\sigma > 0$
\begin{equation}\label{e:remainders}
\lim_{\delta\to 0^+}\limsup\limits_{\ell\to \infty}\sup\limits_{d_g(x,y)\le \delta}\left|F_{\delta,\sigma}(x,y,\ell)\right| = 0.
\end{equation}

Define 
$$
A(\ell,x,y) 
:= \frac{1}{(2\pi)^n\Theta^{\frac{1}{2}}(x,y)}\int\limits_{S_y^*M}e^{i\nu_\ell\langle\exp_y^{-1}(x),\omega\rangle_g}\frac{d\omega}{\sqrt{|g_y|}}.
$$
To deal with the first term in \eqref{e:triple_sum}, we claim that 
\begin{equation}\label{e:L_goal2}
\lim\limits_{\sigma\to 0^+}\lim_{\bad\to 0}\lim_{\delta\to 0}\sup_{d_g(x,y)<\delta}\limsup_{\ell\to \infty}\nu_\ell^{-|\alpha|-|\beta|}\Big|\partial_x^\alpha\partial_y^\beta A_{\bad}(\ell,x,y)
- \partial_x^\alpha\partial_y^\beta A(\ell,x,y)\Big|=0.
\end{equation}
 and
\begin{equation}\label{e:L_goal}
\lim\limits_{\sigma\to 0^+}\lim_{\bad\to 0}\lim_{\delta\to 0}\sup_{d_g(x,y)<\delta}\limsup_{\ell\to \infty}\nu_\ell^{-|\alpha|-|\beta|}\Big|\partial_x^\alpha\partial_y^\beta A_{\bad}(\ell,x,y)\sum\limits_{k\in\Z}\wh f_k(0) 
- \frac{2\pi N}{T}\partial_x^\alpha\partial_y^\beta A_{\bad}(\ell,x,y)\Big|=0.
\end{equation}

Observe that~\eqref{e:L_goal2} follows from the fact that 
$$
\limsup_{\ell\to \infty}\nu_\ell^{-|\alpha|-|\beta|}|\partial_x^\alpha\partial_y^\beta (e^{i\nu_\ell\langle\exp_y^{-1}(x),\omega\rangle_g}(1-c_{\bad}^0(y,\omega)))|\leq C_{\alpha\beta}|(1-c_{\bad}^0(y,\omega))|
$$
and that $\|1-c_{\bad}^0(y,\omega)\|_{L^1(\mathbb{S}^{n-1}_\omega)}\to 0.$

To prove \eqref{e:L_goal}, first note that by \eqref{e:f_k} and \eqref{e:rho_ep} we have
\[
\sum\limits_{k\in\Z}\wh f_k(0) = \sum\limits_{k\in\Z} \lim\limits_{t\to 0}{{e^{i\pi t\frac{N-1}{T}}}}\frac{\sin\lp\frac{\pi N t}{T}\rp}{\sin\lp\frac{\pi t}{T}\rp}\wh\psi_{\sigma}(t+kT)\wh\rho_\delta(t) = N\sum\limits_{k\in\Z}\wh\psi_{\sigma}(kT).
\]
Using the Poisson summation formula and the definition of $\psi_{\w,\sigma}:=\psi_{\sigma}$ in \eqref{e:psi_defn},
\begin{align}
\begin{split}
N\sum\limits_{k\in\Z}\wh\psi_{\sigma}(kT)   &= N\sum\limits_{k\in\Z}\frac{\sin(\w kT)}{kT}\wh\rho(\sigma kT) 
 = \frac{2\pi N}{T}\sum\limits_{k\in\Z}\mathds 1_{[-1,1]}\ast\rho_{\sigma/\w}(2\pi k/T\w)\\
&= \frac{2\pi N}{T}\sum\limits_{k\in\Z}\psi_{1,\sigma/\w}\lp\frac{2\pi k}{T\w}\rp.
\end{split}
\end{align}
\color{black}
Motivated by the form of the above expression, we replace $\sigma$ by $\w\sigma$, which is permitted since $\w$ is fixed throughout this argument. Thus, 
\begin{equation}\label{poisson}
N\sum\limits_{k\in\Z}\wh\psi_{\w,\w\sigma}(kT) = \frac{2\pi N}{T}\sum\limits_{k\in\Z}\psi_{1,\sigma}\lp\frac{2\pi k}{T\w}\rp.
\end{equation}
 Since $\psi_{1,\sigma} =\mathds 1_{[-1,1]}\ast\rho_{\sigma}$ and $0<\w<\frac{2\pi}{T}$, we have that for $k\ne 0,$
\[
\left|\frac{1}{T}\psi_{1,\sigma}\lp\frac{k}{T\w}\rp\right| \le \frac{C_{N'}}{T}\lp1 + \frac{|k|}{T\w\sigma}\rp^{-N'}\quad \text{for any }N'.
\]
Thus, if we choose $N' \ge 2$, we obtain
\[
\left|\sum\limits_{\substack{k\in\Z \\ k\ne 0}}\frac{1}{T}\psi_{1,\sigma}\lp\frac{k}{T\w}\rp\right| \le C_{N'} \sum\limits_{\substack{k\in\Z \\ k\ne 0}}(\w\sigma)^{N'}T^{-1}(\w\sigma+|k|/T)^{-N'} \le C_{N'}(\w\sigma)^{N'}T^{-1}\sum\limits_{\substack{k\in\Z \\ k\ne 0}} (|k|/T)^{-N},
\]
which converges to $0$ as $\sigma\to 0.$ Also, when $k = 0$, we have
\[
\psi_{1,\sigma}(0) = \frac{1}{2\pi}\int\limits_{-\infty}^\infty \frac{2\sin t}{t}\wh\rho(\sigma t)\,dt \to \psi(0) = 1
\]
as $\sigma\to 0,$ and this finishes the proof of the claim in \eqref{e:L_goal}.

Combining \eqref{e:R_ep}, \eqref{e:triple_sum}, \eqref{e:remainders}, and \eqref{e:L_goal} yields that the final step in the proof is to eliminate the factor of $\Theta^{-\frac{1}{2}}(x,y)$ implicit in the definition of $L$. For this, we observe that $\Theta^{-\frac{1}{2}}(x,x) = 1$ and its differential vanishes on the diagonal in $M\times M$. Hence, for small $d_g(x,y)$, we have 
\[
\Theta^{-\frac{1}{2}}(x,y) = 1 + d_g(x,y)^2 G(x,y)
\]
for some smooth, bounded function $G$. Thus, it suffices to show that 
\begin{equation}\label{e:theta_suffices}
\lim\limits_{\delta\to 0^+}\limsup\limits_{\ell\to\infty}\sup\limits_{d_g(x,y)\le\delta}\left|\frac{1}{\nu_\ell^{|\alpha|+|\beta|}}\partial_x^\alpha\partial_y^\beta\lp d_g(x,y)^2\int\limits_{S_y^*M}e^{i\nu_\ell\langle\exp_y^{-1}(x),\omega\rangle}\,\frac{d\omega}{\sqrt{|g_y|}}\rp\right| =0.
\end{equation}
In the case where at most one derivative falls on the factor of $d_g(x,y)^2$, the above statement holds trivially. If two or more derivatives fall on this factor, then at most $|\alpha|+|\beta| - 2$ factors of $\nu_\ell$ can appear from differentiating the integral over $S_y^*M$, and so \eqref{e:theta_suffices} also holds in this case. 

\end{proof}
%
%
%

\section{The contributions of subperiodic loops}

\label{s:subperiodic}

\subsection{Subperiodic Loops near $K_N$}
In this section, we analyse the asymptotic behavior of the contributions to the spectral projector from times which are bounded away from integer multiples of $T$, which are characterized by the quantity
\begin{equation}
\label{e:Bdef}
\mathscr B_{\bad}(\lambda,\sigma;x,y) = \frac{1}{2\pi}\int\limits_{-\infty}^\infty e^{it\lambda}\frac{\sin\lp\frac{\pi N t}{T}\rp}{\sin\lp\frac{\pi t}{T}\rp}\wh\psi_{\sigma}(t)U_t\Ces(x,y)\lp 1 - \sum\limits_{k\in\Z}\wh\rho(t-kT)\rp\,dt.\end{equation}
We can rewrite this as 
\[\mathscr B_{\bad}(\lambda,\sigma;x,y) = \sum\limits_{k\in\Z}\frac{1}{2\pi}\int\limits_{kT}^{(k+1)T} e^{it\lambda}\frac{\sin\lp\frac{\pi N t}{T}\rp}{\sin\lp\frac{\pi t}{T}\rp}\wh\psi_{\sigma}(t)U_t\Ces(x,y)\lp 1 - \wh\rho(t-kT) - \wh\rho(t-(k+1)T)\rp\,dt.\]
Changing variables via $t\mapsto t+kT$, we obtain 
\[\mathscr B_{\bad}(\lambda,\sigma;x,y) = \sum\limits_{k\in\Z}\frac{e^{ikT\lambda}}{2\pi}\int\limits_{0}^{T} e^{it\lambda}\frac{\sin\lp\frac{\pi N t}{T}\rp}{\sin\lp\frac{\pi t}{T}\rp}\wh\psi_{\sigma}(t+kT)U_{t+kT}\Ces(x,y)\lp 1 - \wh\rho(t) - \wh\rho(t-T)\rp\,dt.\]
Similarly to \secref{s:smooth}, we use the fact that we can write 
\begin{equation}\label{eq:Q_k_definition}
U_{t+kT} = e^{-ik\mathfrak b T}(U_t+Q_k(t))
\end{equation}
for some $Q_k(t)$ which is an FIO of order $-\frac{1}{4}-1.$ Thus,
\[\mathscr B_{\bad}(\lambda,\sigma;x,y) = \sum\limits_{k\in\Z}\frac{e^{ikT(\lambda-\mathfrak b)}}{2\pi}\int\limits_{0}^{T} e^{it\lambda}\frac{\sin\lp\frac{\pi N t}{T}\rp}{\sin\lp\frac{\pi t}{T}\rp}\wh\psi_{\sigma}(t+kT)(U_t+ Q_k(t))\Ces(x,y)\lp 1 - \wh\rho(t) - \wh\rho(t-T)\rp\,dt.\]
Note that due to the support properties of $\wh\rho$, we can extend the integral over $[0,T]$ to be performed over the whole real line. Let us define 
\begin{equation}\label{e:gkN}\wh g_{k,N}(t) = \frac{\sin\lp\frac{N\pi t}{T}\rp}{\sin\lp\frac{\pi t}{T}\rp}\wh\psi_{\sigma}(t+kT)\lp 1 - \wh\rho(t) - \wh\rho(t-T)\rp 1_{[0,T]}(t)
\end{equation}
so that 
\begin{equation}
\label{e:buzz}
\mathscr B_{\bad}(\lambda,\sigma;x,y) = \sum\limits_{k\in\Z}\frac{e^{ikT(\lambda-\mathfrak b)}}{2\pi}\mathcal F^{-1}_{t\mapsto\lambda}\left[ \wh g_{k,N}\lp U_t + Q_k(t)\rp \Ces\right].
\end{equation}
\begin{lem}
Suppose that in some coordinate chart, 
\[U_t(x,y) = \frac{1}{(2\pi)^n}\int\limits_{\R^n}e^{i\varphi(x,y,t,\xi)}a(x,y,\xi)\,d\xi\]
for some nondegenerate homogeneous phase function $\varphi$ and some symbol $a\in S^0$. Then,
\[U_t\Ces(x,y) = \frac{1}{(2\pi)^n}\int\limits_{\R^n}e^{i\varphi(x,y,t,\xi)}a_{\bad}(x,y,\xi)\,d\xi,  \]
where
$$
a_{\bad}(x,y,\xi)-a^0(x,y,\xi)\overline{c_{\bad}^0(y,-\partial_y \varphi)}\in S^{-1}.
$$
\end{lem}

\noindent This lemma follows from the standard FIO calculus (c.f. \cite[Ch.\,25]{HormanderBook2009}). With this in hand, we have the following proposition.

\begin{prop}
\label{p:rationalTimes}
Let $(M,g)$ be a smooth, compact, Zoll manifold with minimal common period $T$. Suppose that $x_0\in M$ satisfies
$$
\lim_{\delta\to 0}\mu_{S^*M}(\mathcal{L}_{N,\delta}(x_0))=0,
$$
and let $c_{\bad}$ satisfy~\eqref{e:suppc} and $\mathscr{B}_{\bad}$ be as in~\eqref{e:Bdef}. Then, for all $\alpha, \beta \in \mathbb N$,  $\sigma>0$,
$$
\lim_{\bad\to 0}\lim_{\delta\to 0^+}\limsup_{\lambda \to \infty} \sup_{x,y\in B(x_0,\delta)}\lambda^{1-n-|\alpha|-|\beta|}\Big|\partial_x^\alpha\partial_y^\beta\mathscr B_{\bad}(\lambda,\sigma;x,y) \Big|=0.
$$
\end{prop}

\begin{proof}

We claim that for $\w > 0$ $\sigma>0$, and any $k\in \Z$ fixed, we have

\begin{equation}\label{e:UC_bound}
 F^{-1}_{t\mapsto\lambda}\left[ \wh g_{k,N} U_t \Ces(x,y)\right]
= D_{\bad,k}(\lambda,x,y)\lambda^{n-1} + R_{\bad,k}(\lambda,x,y),
\end{equation}
where $D_{\bad,k}$ and $R_{\bad,k}$ are functions satisfying
\begin{gather*}\lim\limits_{\bad \to 0}\lim_{\delta \to 0}\limsup\limits_{\lambda > 0}\sup\limits_{x,y\in B(x_0,\delta)}\sum_k\lambda^{-|\alpha|-|\beta|}\left|\partial_{x}^\alpha\partial_y^\beta D_{\bad,k}(\lambda,x,y)\right| = 0\, \\  \lim_{\lambda\to \infty}\sup\limits_{x,y\in B(x_0,\delta)}\lambda^{1-n-|\alpha|-|\beta|}\left|\partial_x^\alpha\partial_y^\beta R_{\bad,k}(\lambda,x,y)\right| =0.\end{gather*} Moreover, we claim that if $Q_k$ is as in \eqref{eq:Q_k_definition}, then 
\begin{equation}\label{e:QC_bound}
\lim\limits_{\lambda\to\infty}\sup_{x,y\in B(x_0,\delta)}\left| {\lambda^{1-n-|\alpha|-|\beta|}}\partial_x^\alpha\partial_y^\beta F^{-1}_{t\mapsto\lambda}\left[ \wh g_{k,N} Q_k \Ces(x,y)\right]\right| = 0.
\end{equation}

We start proving the Proposition given the claim. Notice that, since $\sigma>0$, the sum in~\eqref{e:buzz} is finite, we have 
$$
\limsup_{\lambda \to \infty} \sup_{x,y\in B(x_0,\delta)} \lambda^{1-n-|\alpha|-|\beta|}|\partial_x^\alpha\partial_y^\beta \mathscr B_{\bad}(\lambda,\sigma;x,y) |\leq \sum_k\limsup_{\lambda\to \infty}\sup_{x,y\in B(x_0,\delta)}\lambda^{-|\alpha|-|\beta|}|\partial_x^\alpha\partial_y^\beta D_{\bad,k}(\lambda, x,y)|
$$
The proposition then follows since
\begin{multline*}
\lim_{\bad\to 0}\lim_{\delta\to 0^+}\sum_k\limsup_{\lambda\to \infty}\sup_{x,y\in B(x_0,\delta)}\lambda^{-|\alpha|-|\beta|}|\partial_x^\alpha\partial_y^\beta D_{\bad,k}(\lambda, x,y)|\\=\sum_k\lim_{\bad\to 0}\lim_{\delta\to 0^+}\limsup_{\lambda\to \infty}\sup_{x,y\in B(x_0,\delta)}\lambda^{-|\alpha|-|\beta|}|\partial_x^\alpha\partial_y^\beta D_{\bad,k}(\lambda, x,y)|=0.
\end{multline*}

Fix $\bad>0$. Note that since $C_{\bad}$ is a pseudodifferential operator, the canonical relation of $U_t\Ces$ is identical to that of $U_t$, which we denote by 
\[\mathscr C = \{(x,\xi,y,\eta,t,\tau):\,|\tau| = |\xi|_{g_y}, \Phi^t(y,\eta) = (x,\xi)\},\]
and hence $U_t\Ces(x,y)$ can be represented as a locally finite sum of expressions of the form 
\[\frac{1}{(2\pi)^n}\int_{\R^n} e^{i\varphi(x,y,t,\xi)}a_{\bad}(x,y,t,\xi)\,d\xi,\]
where $a_{\bad}^0(x,y,t,\xi) = a^0(x,y,\xi)\overline{c_{\bad}^0(y,-d_y\varphi)}$ with $a^0$ and $c_{\bad}^0$ being the principal symbols of $U_t$ and $C_{\bad}$, respectively. Here, $\varphi$ is some nondegenerate phase function parameterizing $\mathscr C$. Thus, we have that $F^{-1}_{t\mapsto\lambda}\left[ \wh g_{k,N} U_t \Ces(x,y)\right]=\frac{1}{2\pi}\int\limits_{-\infty}^\infty e^{it\lambda} \wh g_{k,N}(t)U_t\Ces(x,y)\,dt$ can be expressed as a locally finite sum of terms of the form 
\begin{align*}
& \frac{1}{(2\pi)^n}\int\limits_{-\infty}^\infty\int\limits_{\R^n} e^{it\lambda + i\varphi(x,y,t,\xi)}\wh g_{k,N}(t)a_{\bad}(x,y,t,\xi)\,d\xi\,dt\\
&\hskip 0.5in= \frac{\lambda^n}{(2\pi)^n}\int\limits_{-\infty}^\infty\int\limits_{\R^n} e^{i\lambda(t + \varphi(x,y,t,\xi))}\wh g_{k,N}(t)a_{\bad}^0(x,y,t,\xi)\,d\xi\,dt + \mathcal O(\lambda^{n-2}),
\end{align*}
since the subprincipal symbol of $C_{\bad}$ is zero in a neighborhood of $x_0$. To see this, observe that the principal symbol is independent of $x$ in a neighborhood of $x_0$ and is homogeneous degree 0 for $|\xi|$ large enough. Let us convert to polar coordinates via $\xi = r\omega$ for $r > 0$ and $\omega \in \R^n$, which gives 

\begin{align*}
& \frac{\lambda^n}{(2\pi)^n}\int\limits_{-\infty}^\infty\int\limits_{\R^n} e^{i\lambda(t + \varphi(x,y,t,\xi))}\wh g_{k,N}(t)a_{\bad}^0(x,y,t,\xi)\,d\xi\,dt \\
&\hskip 0.5in = \frac{\lambda^n}{(2\pi)^n}\int\limits_{-\infty}^\infty\int\limits_{0}^\infty\int\limits_{S^{n-1}} e^{i\lambda(t + r\varphi(x,y,t,\omega))}\wh g_{k,N}(t)a_{\bad}^0(x,y,t,\omega)r^{n-1}\,dr\,dt\,d\omega.
\end{align*}
Let $\chi_{\delta'}\in C^\infty(\mathbb{S}^{n-1})$ be such that $|\nabla_\omega\varphi(x,y,t,\omega)| < \delta'$ for all $\omega\in \supp\chi_{\delta'}.$ Then,

\begin{align}
\begin{split}\label{B_chi_polar}
& \frac{\lambda^n}{(2\pi)^n}\int\limits_{-\infty}^\infty\int\limits_{0}^\infty\int\limits_{S^{n-1}}e^{i\lambda(t + r\varphi(x,y,t,\omega))}\wh g_{k,N}(t)a_{\bad}^0(x,y,t,\omega)r^{n-1}\,dr\,dt\,d\omega \\
& \hskip 0.5in = \frac{\lambda^n}{(2\pi)^n}\int\limits_{-\infty}^\infty\int\limits_{0}^\infty\int\limits_{S^{n-1}}e^{i\lambda(t + r\varphi(x,y,t,\omega))}\chi_{\delta'}(\omega)\wh g_{k,N}(t)a_{\bad}^0(x,y,t,\omega)r^{n-1}\,dr\,dt\,d\omega + \mathcal O(\lambda^{-\infty}),
\end{split}
\end{align}
since $|\nabla_\omega\varphi|$ is bounded below on the support of $1-\chi_{\delta'}$, and so we may integrate by parts arbitrarily many times in $\omega$ using the operator $\frac{\nabla_\omega\varphi\cdot \nabla_\omega}{i\lambda r |\nabla_\omega \varphi|^2}$. Now, for each fixed $\omega$, we aim to perform stationary phase in $(t,r)$. The critical points of the phase function $\wt \varphi = t+r\varphi(x,y,t,\omega)$ occur when 
\[\partial_t\varphi(x,y,t,\omega) = -\frac{1}{r}\quad\text{and}\quad \varphi(x,y,t,\omega) = 0.\]
The Hessian at such points is given by 
\[\lp\begin{array}{cc}
r\partial_t^2\varphi & \partial_t\varphi\\
\partial_t\varphi & 0
\end{array}\rp,\]
and hence the critical points are nondegenerate. Since $\varphi$ is homogeneous of degree 1 in the fiber variable, we have that $\partial_r\varphi(x,y,t,r\omega) = \varphi(x,y,t,\omega) = 0$ at any of the critical values of $t$. Also, since $|\nabla_\omega \varphi(x,y,t,r\omega)| < \delta'$ on $\supp\chi_{\delta'}$, we have that $|d_\xi\varphi| < \delta'$ at the critical points, and hence the points
\[(x,d_x\varphi,y,-d_y\varphi,t,\partial_t\varphi)\]
are very close to the canonical relation $\mathscr C.$ Thus, for $\delta'$ sufficiently small, we have that at each critical point $(t_c,r_c)$, 
\[d_g(\Phi^{t_c}(y,-d_y\varphi) , (x,d_x\varphi)) < \min(\bad,\delta).\]
Thus, if $(x,y)\in B(x_0,\bad/2)$, this implies that $\Phi^{t_c}(y,-d_y\varphi)\in B(x_0,\bad)$. Due to the support properties of $c_{\bad}(y,-d_y\varphi)$ (see~\eqref{e:suppc}), we have that $-d_y\varphi\notin L_{\bad}(x_0)$, (see~\eqref{e:L_def} for the definition of $L_{\bad}$). In particular, the only critical points which contribute a nonzero term to the sum are those for which $|t_c - \frac{p}{q}T| < \bad$ for some $0 < p < q\le N$. Therefore, by stationary phase, the leading term in \eqref{B_chi_polar} can be expressed as a finite sum of terms of the form 
\begin{equation}
\label{e:stationaryPhase}
\begin{aligned}
\frac{\lambda^{n-1}}{(2\pi)^n}\int\limits_{S^{n-1}}\sum\frac{1}{|\partial_t\varphi|}e^{i\lambda t_c + i\frac{\pi}{4}|\sgn \Hess\wt \varphi|}\chi_{\delta'}(\omega)\wh g_{k,N}(t_c)a_{\bad}^0(x,y,t_c,\omega)r_c^{n-1}\,d\omega +\mathcal O_{\delta',\bad}(\lambda^{n-2}),
\end{aligned}
\end{equation}
where the sum is taken over all critical points $(t_c(\omega),r_c(\omega))$ for which $a_{\bad}^0(x,y,t_c,\omega)\ne 0$. Since, for $d_g(x,y)<\bad/2$ small enough, $|t_c-\frac{p}{q}T| < \bad$ for some $0 < p < q\le N$, we have that $|\sin(N\pi t_c/T)|\leq N\pi\bad/T$, and hence   
\begin{multline*}
\lim\limits_{\bad\to 0}\lim_{\delta\to 0}\sup_{d_g(x,y)<\delta} \wh g_{k,N}(t_c) a_{\bad}^0(x,y,t_c,\omega)
\\
= \lim\limits_{\bad\to 0} \lim_{\delta\to 0}\sup_{d_g(x,y)<\delta}\frac{\sin\lp\frac{N\pi t_c}{T}\rp}{\sin\lp\frac{\pi t_c}{T}\rp}\wh\psi_{\sigma}(t_c+kT)\lp 1 - \wh\rho(t_c) - \wh\rho(t_c-T)\rp a_{\bad}^0(x,y,t_c,\omega)= 0,\end{multline*}
which completes the proof of \eqref{e:UC_bound}. The estimate \eqref{e:QC_bound} follows by a similar argument if we note that $Q_k(t)$ is an FIO of one order lower than $U_t$ with the same canonical relation.  

This completes the proof for $|\alpha|=|\beta|=0$. To handle derivatives, observe that if a derivative falls on the amplitude, then the kernel is smaller than the main term by a power of $\lambda$ and hence does not contribute after taking the $\lambda\to \infty$. Therefore, the only term we need to consider is when all of the derivatives fall on the exponential. In this case, we obtain~\eqref{e:stationaryPhase} with $\lambda^{n-1}$ replaced $\lambda^{n-1+|\alpha|+|\beta|}$ and the symbol $a_\w^0$ replaced by another symbol $\tilde{a}_{\w}^0$ with same support properties. Hence, the proof is completed in the same way as for $|\alpha|=|\beta|=0$. 
\end{proof}

\subsection{Looping times outside of $K_N$}
By the assumptions of \thmref{t:main_thm_2}, we know that there must be at most a small measure set of subperiodic loops with lengths which are outside of $K_N = \{\frac{p}{q}T:\, 1\le p < q \le N\}$, and thus we expect their contributions to be negligible. In this section, we demonstrate this rigorously by utilizing analysis similar to \cite{SoggeZelditch2002,CH15}.

\begin{prop}
\label{p:b_eps_prop}
Let $B_{\bad}\in\Psi^0(M)$ be any pseudodifferential operator such that the principal symbol $b_{\bad}^0(x,\xi)$ satisfies
$$
\sup_{x\in M}\|b_{\bad}^0(x,\cdot)\|_{L^1(\mathbb{S}^{n-1})}\leq C\bad.
$$
Then, for any $a > 0$, $\alpha, \beta\in \mathbb N$,  there exist constants $\mu_0, C > 0$ such that 
\[\sup\limits_{x,y\in M}\left|\partial_x^\alpha\partial_y^\beta\Pi_{[\mu,\mu+a]}B_{\bad}^*(x,y)\right|\le C\bad\mu^{n-1+|\alpha|+|\beta|}\]
for all $\mu\ge \mu_0$.
\end{prop}

\begin{rmk}\textnormal{
Note that this proposition holds on \emph{any} Riemannian manifold, not only Zoll manifolds.
}
\end{rmk}

First, we claim that for any $\rho\in\mathscr S(\R)$ with Fourier transform $\wh\rho$ supported in $[-\frac{1}{2}\text{inj}(M,g),\frac{1}{2}\text{inj}(M,g)]$ with $\wh\rho \equiv 1$ in a neighborhood of 0, we have 
\[\Theta^{\frac{1}{2}}(x,y)\partial_\mu\lp\Pi_{[0,\mu]}B_{\bad}^*\rp(x,y) = \frac{\mu^{n-1}}{(2\pi)^n}\int\limits_{S_y^*M}e^{i\mu\langle \exp_y^{-1}(x),\omega\rangle}\overline{b_{\bad}^0(y,\omega)}\frac{d\omega}{\sqrt{|g_y|}} + R_{\bad}(\mu,x,y),\]
where 
\[\left|\partial_x^\alpha\partial_y^\beta R_{\bad}(\mu,x,y)\right| \le C_0 \mu^{n-2+|\alpha|+|\beta|}\]
for some $C_0 > 0.$ This follows by a repetition of the proof of \propref{Q_prop} with $(I+Q)C$ replaced by $B_{\bad}^*$ and $f$ replaced by $\rho.$ Since $\wh\rho \equiv 1$ near 0, the first term in the remainder estimate \eqref{Q_prop_R_est} vanishes. Then, since $b_{\bad}^0$ is supported in a set of Liouville measure zero, we have that 
\begin{equation}
\label{e:B_eps_smooth}
\left|\partial_\mu\partial_x^\alpha\partial_y^\beta\lp \rho\ast \Pi_{[0,\mu]}B_{\bad}^*(x,y)\rp\right| \le c_1{\bad}\mu^{n-1+|\alpha|+|\beta|} + c_{\bad}\mu^{n-2+|\alpha|+|\beta|}
\end{equation}
for sufficiently large $\mu,$ where $c_1 > 0$ is independent of ${\bad}$.  

To control the difference $\partial_\mu \Pi_{[0,\mu]}B_{\bad}^* - \partial_\mu(\rho\ast \Pi_{[0,\mu]}B_{\bad}^*)$, we invoke general Tauberian theorems in the exact same fashion as in \cite{CH15}.  

\begin{lem}[\cite{?}, Tauberian theorem for non-monotone functions]
\label{lem:tauberian}
Let $\theta$ be a piecewise continuous function such that there exists $A > 0$ with $\wh \theta(t) = 0$ for $|t|\le A$. Suppose further that for all $\mu\in\R$ there exist constants $m\in\N$ and $c_1,c_2 > 0$ such that 
\begin{equation}\label{e:taub_hypothesis}
|\theta(\mu+s) - \theta(\mu)| \le c_1(1+ |\mu|)^{m} + c_2(1+|\mu|)^{m-1}\quad \text{for all }s\in [0,1].
\end{equation}
Then, there exists a positive constant $c_{m,A}$, depending only on $m$ and $A$, such that for all $\mu$ we have 
\[|\theta(\mu)|\le c_{m,A}\lp c_1(1+|\mu|)^m + c_2(1+|\mu|)^{m-1}\rp.\]
\end{lem}

To apply this lemma, we first set 
\[\theta_{\bad}(x,y,\mu) = \partial_x^\alpha\partial_y^\beta(\Pi_{[0,\mu]}B_{\bad}^*(x,y) - \rho\ast \Pi_{[0,\mu]}B_{\bad}^*(x,y)).\]
We must then demonstrate that $\theta_{\bad}$ satisfies the hypotheses listed above. Note that 
\[\mathcal{F}_{\mu\to t}\lp \theta_{\bad}(x,y,\cdot)(t) \rp= (1-\hat\rho(t))\mathcal{F}_{\mu\to t}(\partial_x^\alpha\partial_y^\beta(\Pi_{[0,\cdot ]}B_{\bad}^*(x,y) )(t).\]
Since $\wh\rho \equiv 1$ near 0, we therefore have that 
\[\mathcal F_{\mu\mapsto t}(\theta_{\bad}(x,y,\cdot)) = 0\]
for $t$ in some interval around 0.
Then, $\mathcal F_{\mu\mapsto t}(\tilde{\theta}_{\bad}(x,y,\cdot))$ vanishes in a neighborhood of $t=0$. We now verify the hypothesis~\eqref{e:taub_hypothesis} for $\theta_{\bad}$.


So let $s\in[0,1]$ and $\mu\in\R$, and notice that 
\begin{align*}
\theta_{\bad}(x,y,\mu+s) - \theta_{\bad}(x,y,\mu) & = \partial_x^\alpha\partial_y^\beta\lp\Pi_{[0,\mu]}B_{\bad}^*(x,y,\mu+s) - \Pi_{[0,\mu]}B_{\bad}^*(x,y,\mu)\rp \\
& \hskip 0.5in + \partial_x^\alpha\partial_y^\beta\lp\wh\rho\ast \Pi_{[0,\mu]}B_{\bad}^*(x,y,\mu+s) - \wh\rho\ast\Pi_{[0,\mu]}B_{\bad}^*(x,y,\mu)\rp.
\end{align*}
To control the first difference, we apply Cauchy-Schwartz, which gives 
\begin{align*}
&\left|\partial_x^\alpha\partial_y^\beta(\Pi_{[0,\mu]}B_{\bad}^*(x,y,\mu+s) - \Pi_{[0,\mu]}B_{\bad}^*(x,y,\mu))\right| \\
& =  \left|\sum\limits_{\mu < \lambda_j\le\mu+s}\partial_x^\alpha\varphi_j(x)\overline{\partial_y^\beta B_{\bad}\varphi_j(y)}\right|\\
& \le \lp\sum\limits_{\mu < \lambda_j\le\mu+s}|\partial_x^\alpha\varphi_j(x)|^2\rp^{1/2} \lp\sum\limits_{\mu < \lambda_j\le\mu+s}|\partial_y^\beta B_{\bad}\varphi_j(y)|^2\rp^{1/2}.
\end{align*}
Applying the local Weyl law (c.f. \cite[Thm 5.2.3]{SoggeBook2014}), we have 
\[\left|\partial_x^\alpha\partial_y^\beta(\Pi_{[0,\mu]}B_{\bad}^*(x,y,\mu+s) - \Pi_{[0,\mu]}B_{\bad}^*(x,y,\mu))\right| \le c_1 {\bad} \mu^{n-1+|\alpha|+|\beta|} + c_{\bad} \mu^{n-2+|\alpha|+|\beta|}\]
for $\mu \ge 1$ and  $s\in[0,1]$. To estimate the derivatives of $\wh\rho\ast \Pi_{[0,\mu]}B_{\bad}^*(x,y,\mu+s) - \wh\rho\ast\Pi_{[0,\mu]}B_{\bad}^*(x,y,\mu)$, we simply integrate \eqref{e:B_eps_smooth} from $\mu$ to $\mu+s$ to obtain 
\[\left|\partial_x^\alpha\partial_y^\beta(\wh\rho\ast \Pi_{[0,\mu]}B_{\bad}^*(x,y,\mu+s) - \wh\rho\ast\Pi_{[0,\mu]}B_{\bad}^*(x,y,\mu))\right| \le c_1'{\bad}\lambda^{n-1+|\alpha|+|\beta|} + c_{\bad}' \lambda^{n-2+|\alpha|+|\beta|}\]
for all $s\in[0,1]$, all $\mu$ sufficiently large, and some new constants $c_1', c_{\bad}'$ where $c_1' > 0$ is independent of ${\bad}$. Therefore, we have that 
\[\left|\theta_{\bad}(x,y,\mu+s)-\theta_{\bad}(x,y,\mu)\right|\le c_1 {\bad} \mu^{n-1+|\alpha|+|\beta|} + c_{\bad} \mu^{n-2+|\alpha|+|\beta|}\]
after potentially increasing $c_1$ and $c_{\bad}$. Applying \lemref{lem:tauberian} with $m = n-1+|\alpha|+|\beta|$, $A = \frac{1}{2}\text{inj}(M,g)$, and $\theta=\theta_{\bad}$, we obtain, using $n\geq 2$, that
$$
|\theta_{\bad}(\mu,x,y)|\leq c_1\bad\mu^{n-1+|\alpha|+|\beta|}+c_{\bad}\mu^{n-2+|\alpha|+|\beta|},
$$
which completes the proof of \propref{p:b_eps_prop}.

\color{black}
\section{On-diagonal analysis of the spectral projector}\label{s:onDiag}

The goal of this section is to establish a lower bound for the spectral function restricted to the diagonal, which is critical for the purposes of comparing the smoothed projector to the original. In particular, we show that most of the ``mass" of the spectral function is concentrated near 
$$
\bigcup_{\ell\in \N}\big[\nu_\ell - r\ell^{-\frac{1}{2}} , \nu_\ell +r\ell^{-\frac{1}{2}} \big],
$$
with $\nu_\ell$ as defined in \eqref{e:Zoll_eigs}. This is similar to the original eigenvalue clustering result of \cite[Theorem 3.1]{DG75}. We expect that a stronger cluster estimate with $r\ell^{-\frac{1}{2}}$ replaced with $r\ell^{-1}$ should hold, but we do not prove this here as the refined statement is not needed. We also note that the results of this section do not depend on any assumptions about superiodic loops. We need only that all geodesics are periodic with minimal common period $T.$

\begin{prop}\label{clustering_prop}
Let $(M,g)$ be a Zoll manifold with minimal common period $T>0$ and let $\{\varphi_j\}_j$ be the corresponding Laplace eigenfunctions defined in \eqref{e:efxs}.
Let $r>0$ and fix a multi-index $\alpha \in \mathbb N^n$. 
\, Then, there exist $K,C,\lambda_0> 0$ so that for all $x\in M$  and $\lambda \geq \lambda_0$
\begin{equation*}\label{clustering_eqn}
\sum\limits_{\lambda_j\in \mathcal A(K,r,\lambda)}|\partial_x^\alpha \varphi_j(x)|^2 \ge \lp 1 - Cr^{-2}\rp\!\! \sum\limits_{|\lambda_j-\lambda|\le K}|\partial_x^\alpha\varphi_j(x)|^2,
\end{equation*}
where
\begin{equation*}\label{clustering_set}
\mathcal A(K,r,\lambda) = \Big\{\lambda_j:\; |\lambda_j - \lambda| \le K,\, \;\;\lambda_j \in \bigcup_{\ell\in \N}\big[\nu_\ell - r\ell^{-\frac{1}{2}} , \nu_\ell +r\ell^{-\frac{1}{2}}\big]\Big\}.
\end{equation*}
\end{prop}

\begin{proof}
We begin by considering the case where $\alpha = 0$ separately. For this, we proceed in close analogy to the proof of \cite[Theorem 3.1]{DG75}. Let $\chi\in \mathscr S(\R)$ with $\chi \ge 0$ and $\wh\chi\in C_c^\infty(\R)$ with $\wh\chi(0) > 0$. Repeating previous calculations, we have that for $x\in M$
\begin{equation}\label{e:eq1}
\sum\limits_{j=0}^\infty \chi(\lambda - \lambda_j)|\varphi_j(x)|^2 = \frac{1}{2\pi}\int\limits_{-\infty}^\infty e^{it\lambda}\wh\chi(t)U_t(x,x)\,dt.
\end{equation}
Similarly, 
\begin{equation}\label{e:eq2}
\sum\limits_{j=0}^\infty e^{i(\b - \lambda_j)T}\chi(\lambda - \lambda_j)|\varphi_j(x)|^2 
= \frac{1}{2\pi}\int\limits_{-\infty}^\infty e^{it\lambda}\wh\chi(t)e^{i\b T}U_{t+T}(x,x)\,dt.
\end{equation}
Recalling that $U_t - e^{i\b T}U_{t+T}$ is an FIO defined by $\mathcal{C}$ of order $-\frac{1}{4}-1$ (see \eqref{e:canonical_relation}), we know that we can write 
\[
U_t(x,x) - e^{i\b T}U_{t+T}(x,x) = \frac{1}{(2\pi)^n}\int\limits_{T_y^*M}e^{i\phi(t,x,x,\xi)}B(t,x,x,\xi)\,d\xi,
\]
where $B$ is a symbol of order $-1$ and $\phi$ is any admissible phase function which parametrizes $\mathcal{C}$ (c.f. \cite[p. 45]{DG75}). As in the proof of \propref{Q_prop}, we can use the phase function $\phi(t,x,y,\xi) = \langle\exp_{y}^{-1}(x),\xi\rangle_{g_y} - t|\xi|_{g_y}$. Hence,
\begin{align}\label{e:shift_FIO}
&\frac{1}{2\pi}\int\limits_{-\infty}^\infty e^{it\lambda}\wh\chi(t)\Big(U_t(x,x) - e^{i\b T}U_{t+T}(x,x)\Big)\,dt 
= \frac{1}{(2\pi)^{n+1}}\int\limits_{-\infty}^\infty\int\limits_{T_y^*M} e^{it(\lambda-|\xi|)}\wh\chi(t)B(t,x,x,\xi)\,d\xi\,dt \notag\\
&\hspace{5cm}  =\frac{\lambda^n}{(2\pi)^{n+1}}\int\limits_{-\infty}^\infty\int\limits_0^\infty\int\limits_{S_y^*M} \wh\chi(t) e^{i\lambda t(1-s)} s^{n-1} B(t,x,x,\lambda s\omega)\,ds\,d\omega\,dt\notag\\
&\hspace{5cm} =\mathcal O(\lambda^{n-2}).
\end{align}
Here, to obtain the bound in the last line we used the fact that $B$ is a symbol of order $-1$ and repeated the calculations from the proof of \propref{Q_prop} that follow \eqref{e:before_stat}. From \eqref{e:eq1}and \eqref{e:eq2} it follows   that 
\begin{equation}\label{e:diag_shift_eqn}
\sum\limits_{j=0}^\infty \chi(\lambda - \lambda_j)\lp 1  - e^{i(\b-\lambda_j) T}\rp|\varphi_j(x)|^2= \mathcal O(\lambda^{n-2}).
\end{equation}
Thus, 
we can take real parts to obtain that 
\begin{equation}\label{e:real_parts}
\sum\limits_{j=0}^\infty \lp 1 - \cos\lp T( \b - \lambda_j)\rp\rp\chi(\lambda - \lambda_j)|\varphi_j(x)|^2 = \mathcal O(\lambda^{n-2})
\end{equation}
as $\lambda\to\infty.$ 
For any $r,\ell > 0$  define the set 
\[
\mathcal E(\ell,r) = \{\lambda_j\in\Spec(\sqrt{-\Delta_g}):\; r\ell^{-1/2}\le T|\lambda_j-\nu_\ell|\le \pi\}
\]
 Recall that $\nu_\ell = \frac{2\pi \ell}{T} + \b$ by \eqref{e:b}. Thus, if $\lambda_j\in\mathcal E(\ell,r)$, we have that 
\[
1 - \cos\lp T(\b - \lambda_j)\rp= 1 - \cos\Big( T(\nu_\ell - \lambda_j)-2\pi \ell\Big)\ge \tfrac{1}{2}r^2 \ell^{-1} -\tfrac{1}{24}r^4\ell^{-2},
\] 
since  $1 - \cos( \theta- 2\pi  \ell ) \ge \frac{1}{2}\theta^2 -\frac{1}{24}\theta^4$ for $\theta \in [-\pi, \pi]$ and all $\ell \in \N$. 
Therefore, using that $\nu_\ell\geq c\ell$ for $\ell$ large enough, together with \eqref{e:real_parts}, we obtain that for every $r>0$ there exist $C,\ell_0>0$ such that for all $\ell\ge \ell_0$, we have
\begin{align*}
\sum\limits_{\lambda_j\in\mathcal E(\ell,r)}\tfrac{1}{2}r^2\ell^{-1} \min\big(\chi(\mu):\,|\mu|\le \tfrac{\pi}{T}\big)|\varphi_j(x)|^2 & \le C \sum\limits_{\lambda_j\in\mathcal E(\ell,r)}\lp 1 - \cos\lp(\b - \lambda_j)\rp\rp\chi(\nu_\ell-\lambda_j)|\varphi_j(x)|^2 \\
& \le C\ell^{n-2}.
\end{align*}
If we adjust $\chi$ so that $\chi(\mu) > 0$ for all $|\mu|\le \frac{\pi}{T}$, we obtain that 
\begin{equation}\label{e:B_k_eqn}
\sum\limits_{\lambda_j\in\mathcal E(\ell,r)}|\varphi_j(x)|^2 \le Cr^{-2}\ell^{n-1}
\end{equation}
for all $r> 0$ and all $\ell$ large enough. 

Next, observe that for any $K,r > 0$,
\[
\mathcal A(K,r,\lambda)=\{\lambda_j:\,|\lambda_j-\lambda|\le K\}\cap \bigcap\limits_{\ell = 1}^\infty \mathcal E(\ell,r)^c.
\]
Therefore, 
\begin{equation}\label{e:A_sum}
\sum\limits_{\lambda_j\in \mathcal A(K,r,\lambda)}|\varphi_j(x)|^2 = \sum\limits_{|\lambda_j-\lambda|\le K}|\varphi_j(x)|^2 - \sum\limits_{\ell=1}^\infty\sum\limits_{\lambda_j\in\{|\lambda_j-\lambda|\le K\}\cap\mathcal E(\ell,r)}|\varphi_j(x)|^2.
\end{equation}
Note that 
\[
\{\lambda_j:|\lambda_j-\lambda|\le K\}\cap\mathcal E(\ell,r)=\emptyset \qquad \text{if}\;\;\; |\nu_\ell-\lambda|>K+\pi.
\]
Thus, if we define
\[
\mathcal V(\lambda,K) = \{\ell:\,|\nu_\ell-\lambda|\le K+\pi\},
\]
 by \eqref{e:A_sum}
\begin{equation}\label{e:A_sum2}
\sum\limits_{\lambda_j\in \mathcal A(K,r,\lambda)}|\varphi_j(x)|^2 = \sum\limits_{|\lambda_j-\lambda|\le K}|\varphi_j(x)|^2 -
 \sum\limits_{\ell\in \mathcal V(\lambda,K)}\sum\limits_{\lambda_j\in\{|\lambda_j-\lambda|\le K\}\cap\mathcal E(\ell,r)}|\varphi_j(x)|^2.
\end{equation}
In addition, for each $\ell\in \mathcal V(\lambda,K)$, we have that $\nu_\ell\approx \lambda$, and so by \eqref{e:B_k_eqn} that 
\begin{equation}\label{e:EkB}
\sum\limits_{\lambda_j\in\{|\lambda_j-\lambda|\le K\}\cap\mathcal E(\ell,r)}|\varphi_j(x)|^2 \le Cr^{-2}\lambda^{n-1}
\end{equation}
since $\ell\approx\nu_\ell\approx \lambda$. Next, we need the following lemma whose proof we postpone until the end of this section. 

\begin{lem}\label{big_cluster_lemma}
Let $(M,g)$ be any compact smooth manifold of dimension $n$ with Laplace eigenfunctions $\{\varphi_j\}_j$ as in \eqref{e:efxs}. Then, for every multi-index $\alpha \in \N$ there exist $K,C,\lambda_0>0$ so that 
\[
\sum\limits_{|\lambda-\lambda_j|\le K}|\partial_x^\alpha\varphi_j(x)|^2 \ge C\lambda^{n-1+2|\alpha|}
\]
for all $\lambda\ge \lambda_0.$
\end{lem}

Returning to the proof of \propref{clustering_prop}, we can combine \lemref{big_cluster_lemma} with \eqref{e:EkB} to obtain that for $K$ sufficiently large,
\begin{equation}\label{e:almost}
\sum\limits_{\lambda_j\in\{|\lambda_j-\lambda|\le K\}\cap\mathcal E(\ell,r)}|\varphi_j(x)|^2 \le Cr^{-2}\sum\limits_{|\lambda_j-\lambda|\le K}|\varphi_j(x)|^2.
\end{equation}
Furthermore, since the cardinality of $\mathcal V(\lambda,K)$ is proportional to $K$, we can combine \eqref{e:almost} with \eqref{e:A_sum2} to obtain
\[
\sum\limits_{\lambda_j\in \mathcal A(K,r,\lambda)}|\varphi_j(x)|^2 \ge \lp 1 - \frac{C}{r^2}\rp\sum\limits_{|\lambda_j-\lambda|\le K}|\varphi_j(x)|^2,
\]
which completes the proof in case where $|\alpha| = 0$. 

In order to prove the statement for higher order derivatives $\partial_x^\alpha$, one need only show the appropriate analogue of \eqref{e:diag_shift_eqn}. In particular, this will follow from
\begin{equation}\label{diag_shift_deriv}
\frac{1}{2\pi}\int\limits_{-\infty}^\infty e^{it\lambda}\wh\chi(t)\partial_x^\alpha\partial_y^\alpha\lp \Big. U_t(x,y) - e^{i\b T}U_{t+T}(x,y)\rp\big|_{y=x}\,dt = \mathcal O(\lambda^{n-2+2|\alpha|}).
\end{equation}
This follows directly from the off-diagonal analogue of \eqref{e:shift_FIO}, which is given by 
\begin{align*}
&\frac{1}{2\pi}\int\limits_{-\infty}^\infty e^{it\lambda}\wh\chi(t)\lp \Big. U_t(x,y) - e^{i\b T}U_{t+T}(x,y)\rp\,dt\\
& \hskip 0.5in =\frac{\lambda^n}{(2\pi)^{n+1}}\int\limits_{-\infty}^\infty \int\limits_{T_y^*M} e^{i\lambda\lp\langle\exp_y^{-1}(x),\xi\rangle + t(1-|\xi|)\rp}\wh\chi(t)\wh B(t,x,y,\lambda\xi)\,d\xi\,dt.
\end{align*}
Thus, each derivative in $x$ or $y$ yields at most one additional power of $\lambda$, and so by previous arguments we obtain \eqref{diag_shift_deriv}. The rest of the argument proceeds identically to the $|\alpha| = 0$ case. \\
\end{proof}

\begin{proof}[Proof of Lemma \ref{big_cluster_lemma}]
The proof of this lower bound relies on the generalized local Weyl law, which states that if $A$ is a classical polyhomogeneous pseudodifferential operator of order zero, then 
\begin{equation}\label{e:Weyl_pseudo}
A\Pi_{[0,\lambda]} A^*(x,x) = \sum\limits_{\lambda_j\le \lambda}|A\varphi_j(x)|^2 = L_A(x,\lambda)\lambda^n + R_A(\lambda,x),
\end{equation}
where
\[
L_A(x):= C\int\limits_{S_x^*M} |\sigma_0(A)(x,\xi)|^2\,d\xi
\]
for some $C>0$, and $\sup_{x\in M}|R_A(\lambda,x)|\le C_A\lambda^{n-1}$ for some $C_A> 0$ and all $\lambda\ge 1$ (c.f.  \cite[Theorem 5.2.3]{SoggeBook2014}). We note that since $A$ is of order zero, $\left|L_A(x)\right|\le C_A'$ for some $C_A' > 0$. Given these facts, we define for each multi-index $\alpha$ the operator
\[A = \partial_x^\alpha(1+\Delta_g)^{-|\alpha|/2} \in \Psi_{c\ell}^0(M)\]
whose principal symbol is a homogeneous function in $C^\infty(T^*M\setminus 0)$ which can be written in local coordinates as
\[ \sigma_0(A)(x,\xi) = \frac{i^{|\alpha|}\xi^\alpha}{|\xi|_g^{|\alpha|}}.\]
By the local Weyl law, we have 
\begin{align*}
A\Pi_{[\lambda-K,\lambda+K]}A^*(x,x) 
&= \lp A\Pi_{\lambda+K}A^*(x,x) - L_A(x)(\lambda+K)^n \Big.\rp \\
&\;\;\;\; -\lp  A\Pi_{\lambda-K}A^*(x,x) - L_A(x)(\lambda-K)^n \Big.\rp
\\
& \;\;\;\;\;\; +L_A(x)\lp (\lambda +K)^n - (\lambda-K)^n\rp\\
& = R_A(\lambda+K,x) - R_A(\lambda-K,x) + L_A(x)\lp K\lambda^{n-1} + \mathcal O_{K,A}(\lambda^{n-2})\rp.
\end{align*}
Since $|R_A(\lambda,x)| \le C_A\lambda^{n-1}$ and $L_A(x) \ge \delta > 0$ for all $x\in M$ and all $\lambda \ge 1,$ we have that 
\[
A\Pi_{[\lambda-K,\lambda+K]}A^*(x,x)  \ge \lp \delta K - C_A\rp\lambda^{n-1} + \mathcal O_{K,A}(\lambda^{n-2}).
\]
Thus, if we choose $K$ large enough so that $\delta K - C_A > 0, $ there exists a $\lambda_0 > 0$ so that 
\begin{equation}\label{e:pseudo_lowerbound}
A\Pi_{[\lambda-K,\lambda+K]}A^*(x,x) \ge C\lambda^{n-1}
\end{equation}
for some $C> 0$ and all $\lambda \ge \lambda_0$. On the other hand, we can use the functional calculus for $\Delta_g$ to write 
\[
A\Pi_{[\lambda-K,\lambda+K]}A^*(x,x) = \sum\limits_{|\lambda-\lambda_j|\le K}(1+\lambda_j^2)^{-|\alpha|}|\partial_x^\alpha\varphi_j(x)|^2.
\]
Observe that 
\[
\left|\frac{1+\lambda^2}{1+\lambda_j^2} - 1 \right| =  \frac{\left|\lambda^2 - \lambda_j^2\right|}{1 + \lambda_j^2}\le \frac{K(2\lambda+K)}{1 + (\lambda-K)^2},
\]
Since $1 + (\lambda - K)^2 \ge \frac{1}{2}\lambda^{2}$ if $\lambda\ge \frac{1}{4}K$, we obtain
\[\left|\frac{1+\lambda^2}{1+\lambda_j^2} - 1 \right| \le  C K\lambda^{-1} + \mathcal O_{K}(\lambda^{-2})\]
as $\lambda\to\infty$. Using binomial expansion, we also obtain 
\[\left|\frac{(1+\lambda^2)^{|\alpha|}}{(1+\lambda_j^2)^{|\alpha|}} - 1\right| \le C_\alpha K\lambda^{-1} + \mathcal O_{K,\alpha}(\lambda^{-2})\]
for any $\alpha.$ Therefore,
\begin{multline}\label{e:cluster_error}
\left|(1+\lambda^2)^{|\alpha|}A\Pi_{[\lambda-K,\lambda+K]}A^*(x,x) - \sum\limits_{|\lambda_j-\lambda|\le K}|\partial_x^\alpha\varphi_j(x)|^2  \right|\\ \le \Big(C_\alpha K \lambda^{-1} + \mathcal O_{K,\alpha}(\lambda^{-2})\Big)\sum\limits_{|\lambda_j-\lambda|\le K}|\partial_x^\alpha\varphi_j(x)|^2.
\end{multline}
Hence, by~\eqref{e:pseudo_lowerbound},
\begin{equation*}
C\lambda^{n-1+2|\alpha|}\leq (1+\lambda^2)^{|\alpha|}A\Pi_{[\lambda-K,\lambda+K]}A^*(x,x)\leq \Big(1+C_\alpha K \lambda^{-1} + \mathcal O_{K,\alpha}(\lambda^{-2})\Big)\sum\limits_{|\lambda_j-\lambda|\le K}|\partial_x^\alpha\varphi_j(x)|^2
\end{equation*}
Since $C_\alpha K\lambda^{-1}+\mathcal{O}_{K,\alpha}(\lambda^{-2})$ tends to zero as $\lambda \to\infty$ for any fixed $K > 0$, this proves the claim.

\end{proof}

\section{Proof of the main results}\label{s:proofs}
In this section we complete the proof of \thmref{t:main_thm_2} and prove \corref{c:MRW_cor}.

\subsection{Proof of \thmref{t:main_thm_2}}

Let us recall that our goal is to compute the asymptotic behavior of 
\begin{equation}\label{e:sum_proj}
\frac{1}{N}\sum\limits_{j=0}^{N-1}\Pi_{[\nu_{\ell+j}-\w,\nu_{\ell+j}+\w]}(x,y).
\end{equation}
Recalling our pseudodifferential cutoffs $B_{\bad}$ and $C_{\bad}$ as well as the definitions of $\mathscr A_{\bad}$ and $\mathscr B_{\bad}$ in \eqref{e:A_defn} and \eqref{e:B_defn}, respectively, we have shown previously that for any $\sigma > 0$, the smoothed projector satisfies
\begin{align*}
& \rho_\sigma\ast \frac{1}{N}\sum\limits_{j=0}^{N-1}\Pi_{[\nu_{\ell+j}-\w,\nu_{\ell+j}+\w]}(x,y) \\
& \hspace{0.7in}= \frac{1}{N}\rho_\sigma\ast \sum\limits_{j=0}^{N-1}\Pi_{[\nu_{\ell+j}-\w,\nu_{\ell+j}+\w]}(C_{\bad}^*+B_{\bad}^*)(x,y)\\
& \hspace{0.7in}= \frac{1}{N}\mathscr A_{\bad}(\nu_\ell,\sigma,x,y) + \frac{1}{N}\mathscr B_{\bad}(\lambda,\sigma,x,y) + \frac{1}{N}\sum\limits_{j=0}^{N-1}\Pi_{[\nu_{\ell+j}-\w,\nu_{\ell+j}+\w]}B_{\bad}^*(x,y).
\end{align*}
By \propref{p:smooth_prop}, we have that for any multi-indices $\alpha,\beta$, 
\begin{equation}
\label{e:A_final}
\lim\limits_{\sigma\to 0^+}\lim_{\bad\to 0^+}\lim\limits_{\delta\to 0^+}\limsup\limits_{\ell\to \infty}\sup\limits_{d_g(x,y)\le \delta}\left|\frac{1}{\nu_\ell^{n-1+|\alpha|+|\beta|}}\partial_x^\alpha\partial_y^\beta R_{\bad}(\ell,\sigma; x,y)\right| = 0,
\end{equation}
where we recall that $R_{\bad}$ is given by
\[R_{\bad}(\ell,\sigma; x,y):= \mathscr A_{\bad}(\nu_\ell,\sigma;x,y) - \frac{2\pi N}{T}\cdot\frac{\nu_\ell^{n-1}}{(2\pi)^{n}}\int\limits_{S_y^*M}e^{i\nu_\ell\langle\exp_y^{-1}(x),\omega\rangle_g}\frac{d\omega}{\sqrt{|g_y|}}.\]
Additionally, we know by \propref{p:rationalTimes} that 
\begin{equation}
\label{e:B_final}
\lim_{\bad\to 0}\lim_{\delta\to 0^+}\limsup_{\ell\to \infty} \sup_{d_g(x,y)<\delta}\nu_\ell^{1-n-|\alpha|-|\beta|}\Big|\partial_x^\alpha\partial_y^\beta\mathscr B_{\bad}(\nu_\ell,\sigma;x,y) \Big|=0.
\end{equation}
Finally, we have that 
\begin{equation}\label{e:B_eps_final}
\lim\limits_{\bad\to 0}\sup\limits_{x,y\in M}\left|\nu_\ell^{1-n-|\alpha|-|\beta|}\partial_x^\alpha\partial_y^\beta \sum\limits_{j=0}^{N-1}\Pi_{[\nu_{\ell+j}-\w,\nu_{\ell+j}+\w]}B_{\bad}^*(x,y)\right| = 0
\end{equation}
by \propref{p:b_eps_prop}. Therefore, if we combine \eqref{e:A_final}, \eqref{e:B_final}, and \eqref{e:B_eps_final}, we have that the proof of \thmref{t:main_thm_2} reduces to the following lemma.
\begin{lem}
\label{l:finalStep}Suppose that $(M,g)$ is smooth, compact, Zoll manifold with minimal period $T$. Then, for any $w < \frac{\pi}{2T}$ and each pair of multi-indices $\alpha,\beta$, we have
\begin{equation}\label{e:end_goal}
\lim\limits_{\sigma \to 0^+}\limsup\limits_{\ell\to\infty}\nu_\ell^{1-n- |\alpha|-|\beta|} \sup\limits_{x,y\in M}\left|\Big. \partial_x^\alpha\partial_y^\beta\lp\Pi_{[\nu_\ell-\w,\nu_\ell+\w]}(x,y) - \rho_{\sigma}\ast\Pi_{[\nu_\ell-\w,\nu_\ell+\w]}(x,y)\rp\right|=0.
\end{equation}
\end{lem}
Note Lemma~\ref{l:finalStep} this is sufficient to complete the proof because the summation over $j$ in \eqref{e:sum_proj} is finite. Thus, we proceed to prove \eqref{e:end_goal}.

Noting that 
\[
\mathcal F_{\tau\mapsto t}\lp\mathds 1_{[-\w,\w]}(\tau)\rp = \int\limits_{-\w}^\w e^{-it\tau}\,d\tau = \frac{2\sin(t\w)}{t},
\]
we can rewrite
\begin{align}\label{e:difference}
\Pi_{[\lambda-\w,\lambda+\w]}(x,y) - \rho_{\sigma}\ast \Pi_{[\lambda-\w,\lambda+\w]}(x,y)=
\sum\limits_{j=0}^\infty h_{\w,\sigma}(\lambda-\lambda_j)\varphi_j(x)\overline{\varphi_j(y)},
\end{align}
for any $\lambda > 0,$ where
\begin{equation}\label{e:h}
h_{\w,\sigma}(\tau) =  \mathds 1_{[-\w,\w]}(\tau) - \frac{1}{\pi}\int\limits_{-\infty}^\infty e^{it\tau}\wh\rho_{\sigma}(t)\frac{\sin(t\w)}{t}\,dt.
\end{equation}
We claim that $h_{\w,\sigma}$ satisfies a bound of the form 
\begin{equation}\label{e:h_bound}
\left|h_{\w,\sigma}(\tau)\right| \le C_N\Big(1 + \frac{||\tau|-\w|}{\sigma}\Big)^{-N}
 \quad \text{ for any }N \in\N.
\end{equation}
To see this, recall that $\rho$ is a Schwartz-class function with $\int_\R\rho\,dt = \wh\rho(0) = 1$ and $\rho_{\sigma}(\tau) = \frac{1}{\sigma}\rho(\tau/\sigma)$. Thus,
\[
\frac{1}{\pi}\int\limits_{-\infty}^\infty e^{it\tau}\wh\rho_{\sigma}(t)\frac{\sin(t\w)}{t}\,dt 
= \int\limits_{-\w}^\w \frac{1}{\sigma}\rho\lp\frac{\tau-\mu}{\sigma}\rp\,d\mu
= \int\limits_{\frac{\tau-\w}{\sigma}}^{\frac{\tau+\w}{\sigma}}\rho(\mu)\,d\mu.
\]
Suppose $\tau > \w$. Then, 
\[
\left|\int\limits_{\frac{\tau-\w}{\sigma}}^{\frac{\tau+\w}{\sigma}}\rho(\mu)\,d\mu\right| \le \int\limits_{\frac{\tau-\w}{\sigma}}^\infty |\rho(\mu)|\,d\mu \le C_N\Big(1 + \frac{\tau-\w}{\sigma}\Big)^{-N}
\]
for any $N$ since $\rho$ is Schwartz. The analogous estimate clearly holds in the case where $\tau < -\w$. If instead $|\tau| < \w$, then since $\rho$ integrates to 1 and is rapidly decaying, along with the fact that $\mathds 1_{[-\w,\w]}$ is identically one on $[-\w,\w]$, we have that
\[
\left|h_{\w,\sigma}(\tau)\right|
=\left|\mathds 1_{[-\w,\w]}(\tau) - \int\limits_{\frac{\tau-\w}{\sigma}}^{\frac{\tau+\w}{\sigma}}\rho(\mu)\,d\mu\right| \le \int\limits_{-\infty}^{\frac{\tau-\w}{\sigma}}|\rho(\mu)|\,d\mu + \int\limits_{\frac{\tau + \w}{\sigma}}^\infty| \rho(\mu)|\,d\mu\le C_N\Big(1 + \frac{||\tau|-\w|}{\sigma}\Big)^{-N}
\]
for any $N$. Finally, in the case where $|\tau| = \w$, \eqref{e:h_bound} only claims that $h_{\w,\sigma}(\tau)$ is uniformly bounded in $\w,\sigma$, which follows immediately from the fact that 
\[
\left|h_{\w,\sigma}(\w)\right| = \left|1 - \int_{0}^{\frac{2\w}{\sigma}}\rho(\mu)\,d\mu\right| \le 1
\]
along with the analogous statement for $\tau = -\w.$ Therefore, we have proved \eqref{e:h_bound}.

Observe that by \eqref{e:difference} and \eqref{e:h} we have
\begin{align*}
&\left|\partial_x^\alpha\partial_y^\beta\lp\Pi_{[\lambda-\w,\lambda+\w]}(x,y) - \rho_{\sigma}\ast\Pi_{[\lambda-\w,\lambda+\w]}(x,y)\rp\right|\\
& \hspace{1.5in}\le \lp\sum\limits_{j=0}^\infty |h_{\w,\sigma}(\lambda-\lambda_j)||\partial_x^\alpha\varphi_j(x)|^2\rp^{\frac{1}{2}}\lp\sum\limits_{j=0}^\infty |h_{\w,\sigma}(\lambda-\lambda_j)||\partial_y^\beta\varphi_j(y)|^2\rp^{\frac{1}{2}}.
\end{align*}
Thus, the claim in \eqref{e:end_goal} would follow once we prove that given $\alpha \in \mathbb N$, setting $\lambda = \nu_\ell$ gives
\begin{equation}\label{e:h_suffices}
\lim\limits_{\sigma\to 0^+}\limsup\limits_{\ell\to\infty}\frac{1}{\nu_\ell^{n-1 + 2|\alpha|}}\sum\limits_{j=0}^\infty |h_{\w,\sigma}(\nu_\ell-\lambda_j)||\partial_x^\alpha\varphi_j(x)|^2 = 0 .
\end{equation}
 For each $\ell$, decompose $\N=J_1(\ell) \cup J_2(\ell)\cup J_3(\ell)$ with
\begin{align*}
&J_1(\ell):=\{j: |\lambda_j - \nu_\ell| > \tfrac{\pi}{T}\}, \qquad 
J_2(\ell):=\{j: |\lambda_j-\nu_\ell| < r\ell^{-1/2}\}, \\ 
&\qquad \qquad J_3(\ell):=\{j: r\ell^{-1/2}< |\lambda_j - \nu_\ell| \le \tfrac{\pi}{T}\}.
\end{align*}
Note that
\begin{align}\label{e:sum_m}
\sum\limits_{j\in J_1(\ell)} |h_{\w,\sigma}(\nu_\ell-\lambda_j)||\partial_x^\alpha\varphi_j(x)|^2
& = \sum\limits_{m=1}^\infty\sum\limits_{|\lambda_j-\nu_\ell| \in [\frac{m\pi}{T},\frac{(m+1)\pi}{T}]}|h_{\w,\sigma}(\nu_\ell-\lambda_j)||\partial_x^\alpha\varphi_j(x)|^2.
\end{align}
Whenever $|\lambda_j-\nu_\ell| \in [\frac{m\pi}{T},\frac{(m+1)\pi}{T}]$ with $m \ge1$ and $\w < \frac{\pi}{2T}$, we have that 
\begin{equation*}
|h_{\w,\sigma}(\nu_\ell-\lambda_j)| \le C_N \lp 1 + \frac{1}{\sigma}\left|\frac{m\pi}{T} - \w\right|\rp^{-N} \le C_N'\lp  \frac{m}{\sigma
 }\rp^{-N} 
\end{equation*}
for some $C_N' > 0$ by \eqref{e:h_bound}. For the same range of $\lambda_j$, we also have that 
\begin{equation*}
\sum\limits_{|\lambda_j-\nu_\ell|\in[m\pi/T,(m+1)\pi/T]}|\partial_x^\alpha\varphi_j(x)|^2 \le C(1 + \nu_\ell + m\pi/T)^{n-1 + 2|\alpha|}
\end{equation*}
for some $C,C'> 0$ by the local Weyl law \eqref{e:Weyl_pseudo}. Therefore, by \eqref{e:sum_m}
\[
\sum\limits_{j\in J_1(\ell)} |h_{\w,\sigma}(\nu_\ell-\lambda_j)||\partial_x^\alpha\varphi_j(x)|^2\le \wt C_N \sigma^N\sum\limits_{m=1}^\infty (1 + \nu_\ell + m\pi/T)^{n-1+2|\alpha|}m^{-N}
\]
for some $\wt C_N > 0.$ Taking any $N \ge  n +1 + 2\alpha$, we thus obtain
\begin{equation}\label{e:I_eqn}
\sum\limits_{j\in J_1(\ell)} |h_{\w,\sigma}(\nu_\ell-\lambda_j)||\partial_x^\alpha\varphi_j(x)|^2 \le  C_1 \sigma^N \nu_\ell^{n-1+2|\alpha|}
\end{equation}
for some $C_1 > 0$ and any $\sigma> 0$ small.

Next, to estimate the sum over $J_2(\ell)$ we note that for each fixed $r,\w > 0$, one can take $\ell$ sufficiently large so that $|r\ell^{-1/2}-\w| \ge \frac{\w}{2}$, in which case  by \eqref{e:h_bound} that
\[
|h_{\w,\sigma}(\nu_\ell-\lambda_j)| \le C_N \lp1 + \frac{\w}{\sigma}\rp^{-N}\le C_N\lp\frac{\sigma}{\w}\rp^N
\]
for $|\nu_\ell-\lambda_j|\le r\ell^{-1/2}$. By the local Weyl law, we have
\begin{equation}\label{e:II_eqn}
\sum\limits_{j\in J_2(\ell)} |h_{\w,\sigma}(\nu_\ell-\lambda_j)||\partial_x^\alpha\varphi_j(x)|^2 \le C_2\lp\frac{\sigma}{\w}\rp^{\!N}\!\!\nu_\ell^{n-1+2|\alpha|}
\end{equation}
for some $C_2 > 0$ and all $\ell$ sufficiently large.

Finally, to estimate the sum over $J_3(\ell)$ we apply \propref{clustering_prop}, which implies that there exist $K>0$ and  $\ell_0>0$ such that for all $\ell\geq \ell_0$ 
\[
\sum\limits_{j\in J_3(\ell)}|\partial_x^\alpha\varphi_j(x)|^2 \le C r^{-2}\sum\limits_{|\lambda_j-\nu_\ell|\le K}|\partial_x^\alpha\varphi_j(x)|^2 \le C' r^{-2}\nu_\ell^{n-1+2|\alpha|},
\]
where the final inequality follows from the local Weyl law \eqref{e:Weyl_pseudo}. Therefore, since $h_{\w,\sigma}$ is bounded by a uniform constant for all $\w,\sigma > 0,$ we have
\begin{equation}\label{e:III_eqn}
\sum\limits_{j\in J_3(\ell)} |h_{\w,\sigma}(\nu_\ell-\lambda_j)||\partial_x^\alpha\varphi_j(x)|^2  \le C_3 r^{-2}\nu_\ell^{n-1+2|\alpha|}
\end{equation}
for some $C_3> 0$, all $r > 0$, and all $\ell$ sufficiently large. 
Combining \eqref{e:I_eqn}, \eqref{e:II_eqn}, and \eqref{e:III_eqn},
\[
\lim\limits_{\ell\to\infty}\frac{1}{\nu_\ell^{n-1+2|\alpha|}}\sum\limits_{j=0}^\infty |h_{\w,\sigma}(\nu_\ell-\lambda_j)||\partial_x^\alpha\varphi_j(x)|^2 \le C_1\sigma^{N} + C_2(\sigma/\w)^N + C_3 r^{-2}
\]
for all $\w < \frac{\pi}{2T}$ and all $\sigma, r > 0$. Recalling that $\w > 0$ was fixed in the statement of the proposition, we may send $\sigma\to 0$ and $r\to\infty$ to obtain \eqref{e:h_suffices}, which completes the proof.

\qed


\section{Proof of Theorem~\ref{t:main}}\label{s:tmain}

The proof of Theorem~\ref{t:main} follows the same steps as Theorem~\ref{t:main_thm_2}, but is somewhat simpler since the structure of trajectories with period smaller than $T$ is simpler than that of subperiodic loops. In fact, if $\rho\in S^*M$ is periodic with some minimal period $t<T$. Then, $t=\frac{T}{N}$. Since $t>\inj(M)$, this implies $N<\frac{T}{N}$.  

We start by integrating~\eqref{e:lambda*} with respect to $x$ with $C_{\bad}$ replaced by the identity:
\begin{align}
\begin{split}\label{e:lambda*2}
\sum\limits_{j=0}^{N-1}\int \rho_\sigma\ast\Pi_{[\lambda+\frac{2\pi j}{T}-\w,\lambda+\frac{2\pi j}{T}+\w]}(x,x) dx
& = \frac{1}{2\pi}\int\int\limits_{-\infty}^\infty e^{it(\lambda+(N-1)\pi/T)}\frac{\sin\lp\frac{\pi N t}{T}\rp}{\sin\lp\frac{\pi t}{T}\rp}\wh\psi_{\sigma}(t)U_t(x,x)\,dtdx\\
&= \mathscr{A}(\lambda,\sigma)+\mathscr{B}(\lambda,\sigma),
\end{split}
\end{align}
where (similar to~\eqref{e:smooth_sum})
\begin{equation}\label{e:smooth_sum2}
\mathscr A(\lambda,\sigma) = \sum\limits_{k\in\Z}e^{ikT(\lambda-\mathfrak b)}\int \partial_\lambda(f_k\ast\Pi_{[0,\lambda]}(I+Q_k))(x,x)dx
\end{equation}
and (similar to~\eqref{e:buzz})
\begin{equation}
\label{e:buzz2}\mathscr B(\lambda,\sigma) = \sum\limits_{k\in\Z}\frac{e^{ikT(\lambda-\mathfrak b)}}{2\pi}\int\mathcal F^{-1}_{t\mapsto\lambda}\left[ \wh g_{k,N}\lp U_t + Q_k(t)\rp(x,x) \right]dx,\end{equation}
where $g_{k,N}$ is given in~\eqref{e:gkN}.

The term $\mathscr A(\lambda,\sigma)$ is analyzed by integrating~\eqref{e:R_ep} and using the estimate from Proposition~\ref{p:smooth_prop} with $x=y$ and $C_{\bad}=I$ to obtain
$$\lim_{\sigma\to 0^+}\limsup_{\ell\to \infty}\Big(\mathscr A(\nu_\ell,\sigma) - \frac{2\pi N}{T}\cdot\frac{\nu_\ell^{n-1}}{(2\pi)^{n}}\text{vol}(\mathbb{S}^{n-1})\text{vol}_g(M) \Big)=0.
$$

To handle $\mathscr B(\lambda,\sigma) $, we proceed as in Proposition~\ref{p:rationalTimes}, but the analysis is considerably simpler. Let $\chi_{\delta'}\in C^\infty(M\times \omega)$ be such that $|\nabla_{x,\omega}\varphi(x,x,t,\omega)| < \delta'$ for all $(x,\omega)\in \supp\chi_{\delta'}.$  We then arrive at the next formula by following the analysis that led to~\eqref{B_chi_polar}, we obtain that $\int\mathcal F^{-1}_{t\mapsto\lambda}\left[ \wh g_{k,N} U_t(x,x)\right]dx$ can be expressed as a locally finite sum of terms of the form
\begin{align}
\begin{split}\label{B_chi_polar2}
& \frac{\lambda^n}{(2\pi)^n}\int\int\limits_{-\infty}^\infty\int\limits_{0}^\infty\int\limits_{S^{n-1}}e^{i\lambda(t + r\varphi(x,y,t,\omega))}\wh g_{k,N}(t)a^0(x,y,t,\omega)r^{n-1}\,dr\,dt\,d\omega dx \\
& \hskip 0.5in = \frac{\lambda^n}{(2\pi)^n}\int\int\limits_{-\infty}^\infty\int\limits_{0}^\infty\int\limits_{S^{n-1}}e^{i\lambda(t + r\varphi(x,x,t,\omega))}\chi_{\delta'}(x,\omega)\wh g_{k,N}(t)a^0(x,x,t,\omega)r^{n-1}\,dr\,dt\,d\omega dx + \mathcal O(\lambda^{-\infty}).
\end{split}
\end{align}
Then, performing stationary phase in $(r,t)$, we arrive at the analog of~\eqref{e:stationaryPhase}. Then, the expression in \eqref{B_chi_polar2} equals
\begin{align*}
\frac{\lambda^{n-1}}{(2\pi)^n}\int\limits_{S^{n-1}}\sum\frac{1}{|\partial_t\varphi|}e^{i\lambda t_c + i\frac{\pi}{4}|\sgn \Hess\wt \varphi|}\chi_{\delta'}(x,\omega)\wh g_{k,N}(t_c)a^0(x,x,t_c,\omega)r_c^{n-1}\,d\omega dx +\mathcal O_{\delta'}(\lambda^{n-2}),
\end{align*}
where the sum is taken over all critical points $(t_c(x,\omega),r_c(x,\omega))$. Since $|t_c-\frac{p}{q}T|<\w$ for some $0<p<q\leq N$, we obtain
$$
\int\mathcal F^{-1}_{t\mapsto\lambda}\left[ \wh g_{k,N} U_t(x,x)\right]dx =o(\lambda^{n-1}).
$$
The identical argument, but with $U_t$ replaced by $U_tQ_k$ shows that 
$$
\int\mathcal F^{-1}_{t\mapsto\lambda}\left[ \wh g_{k,N} U_tQ_k(x,x)\right]dx =O_k(\lambda^{n-2}).
$$
The proof of Theorem~\ref{t:main} is now completed by Lemma~\ref{l:finalStep} which implies
\begin{equation}\label{e:end_goal2}
\lim\limits_{\sigma \to 0^+}\limsup\limits_{\ell\to\infty}\nu_\ell^{1-n}\left|  \Pi_{[\nu_\ell-\w,\nu_\ell+\w]}(x,x) - \rho_{\sigma}\ast\Pi_{[\nu_\ell-\w,\nu_\ell+\w]}(x,x)\right|=0.
\end{equation}

\bibliography{master_bib}{}
\bibliographystyle{amsalpha}

\end{document}